\documentclass[hidelinks]{siamart}
\usepackage{graphicx}
\usepackage{amsmath,amsfonts,amssymb,cases,mathdots,color,colordvi,url}
\usepackage{geometry}
\usepackage{hyperref}
\usepackage{mathptmx}
\usepackage{color}

\def\V{{\mathbb{V}}}
\def\R{{\mathbb{R}}}

\def\S{{\mathbb{S}}}
\def\T{{\mathbb{T}}}

\begin{document}

\title{Spectral operators of matrices: semismoothness and characterizations of the generalized Jacobian 
}


\author{Chao Ding%
	\thanks{Institute of Applied Mathematics, Academy of Mathematics and Systems Science, Chinese Academy of Sciences, Beijing,  P.R. China (\email{dingchao@amss.ac.cn}). The research of this author was supported  by the National Natural Science Foundation of China under projects
		No. 11671387 and No. 11531014.}
	\and
	Defeng Sun%
	\thanks{Department of Applied Mathematics, The Hong Kong Polytechnic University, Hong Kong  (\email{defeng.sun@polyu.edu.hk}). The research of this author was supported in part by a start-up research grant from the Hong Kong
Polytechnic University.}
	\and
Jie Sun%
	\thanks{Department of Mathematics and Statistics, Curtin University, Australia (\email{Jie.Sun@curtin.edu.au}).}%
	\and
Kim-Chuan Toh%
	\thanks{Department of
Mathematics,  and Institute of Operations Research and Analytics,
 National University of Singapore,
Singapore (\email{mattohkc@nus.edu.sg}).}%
}


\maketitle

\begin{abstract}
Spectral operators of matrices proposed recently in [C. Ding, D.F. Sun, J. Sun, and K.C.  Toh, Math. Program. {\bf 168}, 509--531 (2018)] are a class of matrix valued functions, which map matrices to matrices by applying a vector-to-vector function to all eigenvalues/singular values of
 {the underlying} matrices. Spectral operators play a crucial role in the study of various applications involving matrices such as 
matrix optimization problems (MOPs)  {that include semidefinite programming
as one of the most important example classes}. In this paper, we will  study more fundamental first- and second-order properties of spectral operators, including the Lipschitz continuity, $\rho$-order B(ouligand)-differentiability ($0<\rho\le 1$), $\rho$-order G-semismoothness ($0<\rho\le 1$), and characterization of generalized Jacobians.
\end{abstract}



\section{Introduction}
\label{section:introduction}

Spectral operators of matrices introduced recently in \cite{DSSToh18} are a class of matrix valued functions defined on 
 {a given}
real Euclidean vector space ${\cal X}$ of real/complex matrices over the scalar field of real numbers $\R$.
 {Unlike
 the well-studied classical  matrix functions} \cite[Chapter 9]{GVanLoan12}, \cite[Chapter 6]{HJohnson94}, \cite{Bhatia97,Higham08,HBenisrael73}, which are L\"{o}wner's operators  {generated by applying a single-variable function to each of the eigenvalues/singular values of the underlying matrices},
  {the} spectral operators introduced in \cite{DSSToh18} generate matrix valued functions by applying a vector-to-vector function to all eigenvalues/singular values of  {the underlying}
 matrices (see Definition \ref{def:def-spectral-op} for details).

Besides its intrinsic theoretical interest in linear algebra, spectral operators play a crucial role in the study of a class of optimization problems 
 {known as}
matrix optimization problems (MOPs), which include many important problems such as  matrix norm approximation, matrix completion, rank minimization, graph theory, machine learning, and etc.  \cite{GT94,Toh97,TT98,RFP07,LOverton96,CanRec08,CanTao09,CLMW09,CSPW09,WMGR09,CFP03,Lovasz79,Dobrynin04,KLVempala97,FPSTseng13,LZLi15,LZLu17,ZSo17,LLCDai18}. In particular, for a given unitarily invariant proper closed convex function $f:{\cal X}\to (-\infty,\infty]$, the spectral operator  {that is closely} related to MOPs is the proximal mapping \cite{Rockafellar70} of $f$ at $X$, which is defined by
\begin{equation}\label{eq:def-proximal}
P_{f}(X):={\rm argmin}_{Y\in{\cal X}}\left\{ f(Y)+\frac{1}{2}\|Y-X\|^{2}\right\}, \quad X\in{\cal X},
\end{equation}
where ${\cal X}$ is either the real vector subspace $\S^m$ of $m\times m$ real symmetric or complex Hermitian matrices, or the real vector subspace $\V^{m\times n}$ of $m\times n$ (assume $m\le n$) real/complex matrices. Among different MOP applications, semidefinite programming (SDP) \cite{Todd01}
 {is arguably
one of the most influential 
classes of}
problems and its importance has been well-recognized by researchers even beyond the optimization community. Recent exciting progress has been made  {both in the design of} efficient numerical methods for solving large scale SDPs \cite{ZSToh10,YSToh15} and 
 {in the study of}
second-order variational analysis of SDP problems \cite{DSYe14,Sun06,CSun08,MNRockafellar15}, in which the first- and second-order properties of the special spectral operator, the projection operator over the positive semidefinite matrix cone \cite{SSun02,SSun08}, have played an essential role.  However, for the general MOPs arising recently from different fields, the classical theory developed for L\"{o}wner's operators has become inadequate  {to cope with the new theoretical developments and needs}.
 {Beyond the spectral operators of matrices arising from} proximal mappings, more general spectral operators  {indeed have} played a pivotal role in many other MOP applications \cite{MSPan12}. Therefore, the study of the general spectral operators will provide the necessary foundations for both computational and theoretical study of the general MOPs.  In particular, the first- and second-order properties of spectral operators obtained in \cite{DSSToh18} including the well-definedness, continuity, directional differentiability, and Fr\'{e}chet-differentiability are of fundamental importance in the study of MOPs \cite{Ding12,LSToh12,CLSToh12,CDZhao17}.

In this paper, we will follow the path set in  \cite{DSSToh18} to conduct  extensive  {theoretical}
 studies on spectral operators.  More first- and second-order properties of spectral operators will be discussed in  {depth. These include} the Lipschitz continuity, $\rho$-order B(ouligand)-differentiability ($0<\rho\le 1$), $\rho$-order G-semismoothness ($0<\rho\le 1$), and characterization of generalized Jacobians. In particular, we will study the semismoothness \cite{Mifflin77,QSun93} of spectral operators, which is one of  {the most} important properties for both algorithm design and theoretical study of the general MOPs. Historically, the semismoothness of vector-valued functions  {had} played a crucial role in constructing nonsmooth and smoothing Newton method for nonlinear equations and related problems. In fact, it is shown
  {in} \cite{QSun93,Qi93,PQi93} that  the (strong) semismoothness is the key property for the local (quadratic) superlinear convergence of the Newton method. Nowadays the semismooth Newton method has became one of  {the most important techniques} in optimization \cite{IKSolodov13,UUBrztzke17,ZSToh10,YSToh15,LSToh18,LSToh18b,YSToh18}. In particular, the several semismooth Newton based  methods  {have been} proposed for solving  {various} large-scale optimization problems
  {in  machine learning} applications such as the lasso, fused lasso and convex clustering problems,
  {and they have significantly outperformed} a number of state-of-the-art solvers in terms of
  efficiency  and robustness \cite{LSToh18,LSToh18b,YSToh18}. For MOPs, the semismoothness of the special spectral operator: the projection operator over the SDP cone,
   {has played a key role in the development of the semismooth Newton based
  augmented Laggrangian method implemented in} the software package SDPNAL \cite{ZSToh10} and its enhanced version SDPNAL+ \cite{YSToh15} for solving large-scale SDP problems. Therefore,
 {based on   these recent progress}, we believe that the results on the semismoothness of spectral operators obtained in this paper will lay a  foundation for the research on general MOPs. For the proximal mapping \eqref{eq:def-proximal}, one can obtain its semismooth property by employing the  results
 {recently developed based}
on semi-algebraic geometry \cite{BCRoy87,Coste99}. It is shown in \cite{BDLewis09,Ioffe08}  {that}
 {locally Lipschitz continuous tame functions} (e.g., the proximal mapping \eqref{eq:def-proximal}) are semismooth. For more recent developments on semi-algebraic geometry in optimization, see \cite{ABSvaiter13,DIoffe15,DILewis15,CHien15,LMPham15,BPauwels16,BHPauwels18} and  {the} references therein.  It is worth to note that unlike our approach, by  {just} employing its tameness, one may not able to obtain  {the explicit formulas of the directional derivative and},   {more importantly},  the strong semismoothness of the proximal mapping (see Section \ref{section:semismoothness} for details).

Another fundamental property, which we will study, is the characterization of the Clarke generalized Jacobians \cite{Clarke83} of the locally Lipschitz continuous spectral operators. This is an important theoretical topic in the second-order variational analysis, which is crucial for the study of many perturbation properties of MOPs such as the strong regularity \cite{PSSun03,Sun06,CSun08}, and full and tilt stability \cite{MNRockafellar15,MRockafellar12}. In addition, for the software packages SDPNAL and SDPNAL+, due to the explicit characterization of the Clarke generalized Jacobian of  the projection operator over the positive semidefinite matrix cone, it becomes possible to  {exploit} the second order sparsity of the SDP problems  {inherited from}  the sparse structure of the generalized Jacobian of the reformulated semismooth equations. The second order sparsity  {can substantially reduce} the computational cost of solving the resulting linear systems
 {associated with the semismooth Newton directions. Indeed the efficient computation of
the semismooth Newton directions is} one of the biggest computational challenges in designing efficient
 {second-order} numerical methods for solving large-scale problems.  {To summarize,} we believe that the fundamental results  {obtained in this paper}, especially the second-order properties such as the semismoothness and the  Clarke generalized Jacobian of spectral operators, are of  importance in both  {the} computational and theoretical study of general MOPs.

%


The remaining parts of this paper are organized as follows. In
Section \ref{subsec:nonsymmetric}, we  {briefly} review several preliminary properties of spectral operators of matrices. We study
the Lipschitz continuity and Bouligand-differentiability of spectral operators defined on  {a} single matrix space $\V^{m\times n}$ in Sections \ref{section:Lip} and \ref{section:B-diff}, respectively. Then, the $G$-semismoothness and  characterization of the Clarke generalized Jacobians  of spectral operators are presented in Section \ref{section:semismoothness} and \ref{section:Clake g-J}. In Section \ref{section:extension}, we extend the corresponding results to spectral operators defined on the Cartesian product of several matrix spaces and the smoothing spectral operators. 
We make some final remarks in Section \ref{sect:remarks}.

Below are some common notations and symbols to be used  {later in the paper}:
\begin{itemize}
\item For any $X\in\V^{m\times n}$, we denote by $X_{ij}$ the $(i,j)$-th entry of $X$ and $x_{j}$ the $j$-th column
of $X$.  Let ${I}\subseteq \{1,\ldots, m\}$ and  ${J}\subseteq \{1,\ldots, n\}$ be two index sets. We use
 $ X_{J}$ to  denote the sub-matrix of $X$ obtained by  removing all the columns of $X$ not in   $J$ and $X_{{I}{J}}$ to
  denote the $|I|\times|J|$ sub-matrix of $X$ obtained  by removing all the rows of $X$ not in $I$ and all the columns of $X$ not in  $J$.

\item For  $X\in\V^{m\times m}$, ${\rm diag}(X)$ denotes the column vector   consisting of  all the
diagonal entries of $X$ being arranged from
the first to the last. For  $x\in\R^m$,
${\rm Diag}(x)$  denotes  the {$m\times m$} diagonal matrix whose $i$-th diagonal entry is $x_i$, $i=1,\ldots,m$.
\item We use $``\circ"$ to denote the usual Hadamard product between two matrices, i.e., for any two matrices $A$ and $B$ in $\V^{m\times n}$ the $(i,j)$-th entry of $  Z:= A\circ B \in \V^{m\times n}$ is
	 {$Z_{ij}=A_{ij} B_{ij}$.}
	
	 \item For any $X\in\S^m$, we use $\lambda:\S^m\to \R^m$ to denote the mapping of the ordered eigenvalues of a Hermitian matrix $X$ satisfying $\lambda_1(X)\ge \lambda_2(X)\ge \ldots \ge \lambda_m(X)$. For any $X\in\V^{m\times n}$, we use $\sigma:\V^{m\times n}\to \R^m$ to denote  the mapping of the ordered singular values of $X$ satisfying $\sigma_1(X)\ge \sigma_2(X)\ge \ldots \ge \sigma_m(X)\ge 0$.

	\item Let ${\mathbb O}^{p}$ ($p=m,n$) be the set of  $p\times p$ orthogonal/unitary matrices.
	 {We
	denote $\mathbb{P}^{p}$ and $\pm\mathbb{P}^{p}$ to be} the sets of all $p\times p$ permutation matrices and signed permutation matrices, respectively.  For any $Y\in\S^m$ and $Z\in\V^{m\times n}$, we use ${\mathbb O}^{m}(Y)$ to denote the set of all  orthogonal  matrices whose columns form an orthonormal basis of eigenvectors of $Y$, and use ${\mathbb O}^{m,n}(Z)$ to denote the set of all pairs of orthogonal matrices $(U,V)$, where the columns of $U$ and $V$ form a compatible set of orthonormal left and right singular vectors for $Z$, {respectively}.

\end{itemize}


\section{Spectral operators of matrices}
\label{subsec:nonsymmetric}

The general spectral operators of matrices introduced by \cite{DSSToh18} are defined on the Cartesian product
of several  real or complex matrix spaces. In order to  {summarize} the properties of spectral operators, we first introduce some definitions and notations, which are needed in the subsequent analysis.

Let $s$ be a positive integer and $0\leq s_{0} \leq s$ be a nonnegative integer. For given positive integers $m_{1},\ldots,m_{s}$ and $n_{s_{0}+1},\ldots,n_{s}$, define the real vector space ${\cal X}$ by
\begin{equation}\label{eq:space-X}
{\cal X}:=\S^{m_{1}} \times\ldots\times\S^{m_{s_{0}}}\times\V^{m_{s_{0}+1}\times n_{s_{0}+1}} \times\ldots\times\V^{m_{s}\times n_{s}}.
\end{equation}
Without loss of generality, we assume that $m_{k}\leq n_{k}$, $k=s_{0}+1,\ldots,s$. For any $X=(X_1,\ldots,X_s)\in{\cal X}$, we have for $1\le k\le s_0$, $X_k\in\S^{m_k}$  and $s_0+1\le k\le s$, $X_k\in\V^{m_k\times n_k}$.
Denote
	\begin{equation}\label{eq:space-Y}
	{\cal Y}:=\R^{m_1}\times\ldots\times\R^{m_{s_0}}\times\R^{m_{s_0}}\times\ldots\times \R^{m_s}.
	\end{equation}
For any $X\in{\cal X}$, define $\kappa(X)\in{\cal Y}$ by $\kappa(X):=\left(\lambda(X_1),\ldots,\lambda(X_{s_0}),\sigma(X_{s_0+1}),\ldots,\sigma(X_s)\right)$.
Define the set ${\cal P}$ by $$	{\cal P}:=\left\{\left(Q_{1}, \ldots,Q_{s}\right)\mid Q_{k}\in\mathbb{P}^{m_{k}},\ 1\leq k\leq s_{0}\ {\rm and}\ Q_{k}\in\pm\mathbb{P}^{m_{k}},\ s_{0}+1\leq k\leq s\right\}.$$
Let $g:{\cal Y}\to{\cal Y}$ be a given mapping. For any    $x=(x_1,\dots,x_s)\in  {\cal Y}$ with $x_k\in\R^{m_{k}}$,   we write $g(x)\in{\cal Y}$ in the form $g(x)=\left(g_{1}(x), \ldots,g_{s}(x)  \right)$ with  $g_{k}(x)\in\R^{m_{k}}$ for $1\le k\le s$.
\begin{definition}\label{def:mixed_symmetric}
		The given mapping $g:{\cal Y}\to{\cal Y}$  is said to be {\it mixed symmetric}, with respect to ${\cal P}$, at     $x=(x_1,\dots,x_s)\in {\cal Y}$ with $x_k\in\R^{m_{k}}$, if
		\begin{equation}\label{eq:def-symmetric}
			g(Q_1x_1,\ldots,Q_sx_s)=\left(Q_1g_1(x),\ldots,Q_{s}g_{s}(x)\right)\quad \forall\,\left(Q_{1}, \ldots,Q_{s}\right)\in{\cal P} .
		\end{equation}
		The mapping $g$ is said to be mixed symmetric,  with respect to ${\cal P}$,  over a set ${\cal D} \subseteq {\cal Y}$ if \eqref{eq:def-symmetric} holds for every $x\in{\cal D}$. We call $g$ a {\it mixed symmetric} mapping,  with respect to ${\cal P}$,  if (\ref{eq:def-symmetric}) holds for every $x\in {\cal Y}$.
	\end{definition}

	Note that for each $k\in\{1,\ldots,s\}$, the function value $g_k(x)\in\R^{m_k}$ is dependent on all $x_1,\ldots,x_s$. {When there is no danger of confusion}, in later discussions  we often   drop
	 {the phrase} ``with respect to ${\cal P}$" from    Definition \ref{def:mixed_symmetric}.
	Let ${\cal N}$ be a given nonempty set in ${\cal X}$. Define $\kappa_{\cal N}: =\left\{ \kappa(X)\in{\cal Y}\mid   X \in {\cal N}\right\}$. The following definition of the spectral operator with respect to  {a}  mixed symmetric mapping $g$ is given by \cite[Definition 1]{DSSToh18}.
	\begin{definition}\label{def:def-spectral-op}
		Suppose that  $g:{\cal Y}\to{\cal Y}$  is   mixed  symmetric on $\kappa_{\cal N}$.
		The spectral operator $G:{\cal N}\to{\cal X}$ with respect to   $g$ is defined as $G(X):=\left(G_{1}(X),\ldots,G_{s}(X)\right)$ for $X=(X_1,\ldots,X_s)\in{\cal N}$  {such that}
		\begin{equation*}\label{eq:def-Gk-two-parts}
			G_{k}(X):=\left\{  \begin{array}{ll}
				P_{k}{\rm Diag}\big(g_{k}(\kappa(X))\big)P_{k}^\T  &  \ \ \mbox{if $1\le k\le s_{0}$,}
				\\[3pt]
				U_{k}\left[{\rm Diag}\big(g_{k}(\kappa(X))\big)\quad 0\right]V_{k}^\T  &\ \ \mbox{if $s_{0}+1\le k\le s$,}
			\end{array} \right.
		\end{equation*} where $P_{k}\in{\mathbb O}^{m_{k}}(X_{k})$, $1\leq k\leq s_{0}$, $(U_{k},V_{k})\in{\mathbb O}^{m_{k},n_{k}}(X_{k})$, $s_{0}+1\leq k\leq s$.
	\end{definition}
	
For the well-definedness, continuity and F(r\'{e}chet)-differentiability of spectral operators, one may refer to \cite{DSSToh18} for  details. It is worth mentioning that for the case that ${\cal X}\equiv\S^m$ (or $\V^{m\times n}$) and $g$ has the form $g(y)=(h(y_1),\ldots,h(y_m))\in\R^m$ with $y_i\in\R$ for some given scalar valued   {function} $h:\R\to\R$, the corresponding spectral operator $G$ is just the  L\"{o}wner operator
 {coined in}  \cite{SSun08}  in {recognition} of L\"{o}wner's original contribution on this topic in  \cite{Lowner34} (or the L\"{o}wner non-Hermitian operator \cite{Yang09} if $h(0) =0$). In \cite{Yang09}, Yang studied several important first and second order properties of the L\"{o}wner non-Hermitian operator, including its F-differentiability and the explicit derivative formula (the equivalent form also can be found in \cite{Noferini17}).

	
	Next, we will focus on the study of spectral operators for the case that ${\cal X}\equiv\V^{m\times n}$. The corresponding extensions for  the spectral operators defined on the general Cartesian product  of several matrix spaces will be presented  in Section \ref{section:extension}.
	
	Let ${\cal N}$ be a given nonempty open set in $\V^{m\times n}$. Suppose that $g:\R^m\to\R^m$ is  mixed symmetric with respect to ${\cal P}\equiv\pm{\mathbb P}^m$ (i.e., absolutely symmetric), on an open set $\hat{\sigma}_{ {\cal N}}$ in $\R^{m}$ containing $\sigma _{\cal N}:=\left\{\sigma(X)\mid X\in{\cal N}\right\}$. The spectral operator $G:{\cal N}\to\V^{m\times n}$ with respect to $g$  defined in Definition \ref{def:def-spectral-op} then takes the form of
	$$
	G(X)=U\left[{\rm Diag}(g(\sigma(X)))\quad 0\right]V^\T,\quad  X\in{\cal N},
	$$
	where $(U,V)\in{\mathbb O}^{m,n}(X)$. For  {a} given  $\overline{X}\in{\cal N}$, consider the singular value decomposition (SVD)    {of} $\overline{X}$, i.e.,
	\begin{equation}\label{eq:Y-eig-Z-SVD}
		\overline{X}=\overline{U}\left[ \Sigma(\overline{X})\quad  0 \right]\overline{V}^\T  ,
	\end{equation} where $\Sigma(\overline{X})$ is an $m\times m$  diagonal matrix whose $i$-th diagonal entry is $\sigma_i(\overline{X})$, $\overline{U}\in{\mathbb O}^m$ and $\overline{V}=\left[ \overline{V}_{1} \quad  \overline{V}_{2} \right] \in{\mathbb O}^{n}$ with
	$\overline{V}_{1}\in\V^{n\times m}$ and
	$\overline{V}_{2}\in\V^{n\times (n-m)}$.
	
We end this section by further introducing some necessary notations and results, which are used in later discussions. Let $\overline{\sigma}:=\sigma(\overline{X})\in\R^{m}$. We  use $\overline{\nu}_{1}>\overline{\nu}_{2}>\ldots>\overline{\nu}_{r}>0$ to denote the nonzero distinct singular values of $\overline{X}$.
	Let $a_l$, $l=1,\ldots,r$, $a$, $b$ and $c$ be the index sets defined by
	\begin{equation}\label{eq:def-a-b-c-ak-nonsymmetric}
	\begin{array}{l}
		a_l:=\{i\mid\sigma_{i}(\overline{X})=\overline{\nu}_{l}, \ 1\le i\le m\}, \quad l=1,\ldots,r, \qquad a:=\{i \mid \sigma_{i}(\overline{X})>0,\ 1\le i\le m\} , \\[3pt]
		b:= \{i\mid \sigma_{i}(\overline{X})=0, \ 1\le i\le m\} \quad  {\rm and} \quad c:=\{m+1,\ldots,n\} .
	\end{array}
	\end{equation}
Denote $\bar{a}:=\{1,\ldots,n\}\setminus a$. For each $  i\in \{1,\ldots,m\}$, we also define $l_{i}(\overline{X})$ to be the number of singular values which are equal to $\sigma_ i(\overline{X})$ but are ranked  before $i$ (including $i$), and $\tilde{l}_{i}(\overline{X})$ to be the number of singular values which are equal to $\sigma_i(\overline{X})$ but are ranked  after $i$ (excluding $i$),  i.e.,
	define $l_{i}(\overline{X})$ and  $\tilde{l}_{i}(\overline{X})$  such that
	\begin{eqnarray}
		&&\sigma_{1}(\overline{X})\geq\ldots\geq\sigma_{i-l_{i}(\overline{X})}(\overline{X})>\sigma_{i-l_{i}(\overline{X})+1}(\overline{X})=\ldots=\sigma_{i}(\overline{X})=\ldots=\sigma_{i+\tilde{l}_{i}(\overline{X})}(\overline{X})\nonumber \\[3pt]
		&>&\sigma_{i+\tilde{l}_{i}(\overline{X})+1}(\overline{X})\geq\ldots\geq\sigma_{m}(\overline{X}) .\label{eq:onsymmetric-l_i}
	\end{eqnarray}
	In later discussions, when the dependence of $l_{i}$ and $\tilde{l}_{i}$ on $\overline{X}$
is clear  from the context, we often drop $\overline{X}$ from these notations for   convenience. We define two linear  matrix operators
	$S :\V^{p\times p}\to \S^{p}$, $T :\V^{p\times p}\to\V^{p\times p}$  by
	\begin{equation}\label{eq:maps-ST}
		S(Y):=\frac{1}{2}(Y+Y^\T ),  \quad
		T(Y):=\frac{1}{2}(Y-Y^\T ), \quad Y\in\V^{p\times p} .
	\end{equation}
For any given $X\in{\cal N}$, let $\sigma=\sigma(X)$. For the mapping $g$, we define three matrices ${\cal E}^0_{1}({\sigma}),{\cal E}^0_{2}({\sigma})\in\R^{m\times m}$ and ${\cal F}^0({\sigma})\in\R^{m\times(n-m)}$ (depending on $X\in{\cal N}$) by
	\begin{eqnarray}
		({\cal E}^0_{1}({\sigma}))_{ij} &:=&
		\left\{ \begin{array}{ll}  \displaystyle{(g_i(\sigma)-g_j({\sigma}))/(\sigma_{i}-\sigma_{j})} & \mbox{if $\sigma_{i}\neq\sigma_{j}$} ,\\[3pt]
			0 & \mbox{otherwise} ,
		\end{array} \right.  \quad i,j\in\{1,\ldots,m\} ,
		\label{eq:def-matric-E1}
		\\
		({\cal E}^0_{2}({\sigma}))_{ij} &:=&
		\left\{ \begin{array}{ll}  \displaystyle{(g_i({\sigma})+g_j({\sigma}))/(\sigma_{i}+\sigma_{j}}) & \mbox{if $\sigma_{i}+\sigma_{j}\neq 0$} ,\\[3pt]
			0 & \mbox{otherwise} ,
		\end{array} \right.  \quad i,j\in\{1,\ldots,m\} ,
		\label{eq:def-matric-E2}
		\\
		({\cal F}^0({\sigma}))_{ij} &:=&
		\left\{\begin{array}{ll}
			\displaystyle{g_i({\sigma})/\sigma_{i}} & \mbox{if $\sigma_{i}\neq 0$} ,\\[3pt]
			0 & \mbox{otherwise},
		\end{array}\right. \quad  i\in\{1,\ldots,m\},\quad j\in\{1,\ldots, n-m\} .
		\label{eq:def-matric-F}
	\end{eqnarray}
	When the dependence of ${\cal E}^0_{1}({\sigma})$, ${\cal E}^0_{2}({\sigma})$ and ${\cal F}^0({\sigma})$  on $\sigma$ is clear from the context, we often drop ${\sigma}$ from these notations. In particular, let $\overline{\cal E}^0_{1}$, $\overline{\cal E}^0_{2}\in\V^{m\times m}$ and $\overline{\cal F}^0\in\V^{m\times(n-m)}$ be the matrices defined by (\ref{eq:def-matric-E1})-(\ref{eq:def-matric-F}) with respect to $\overline{\sigma}=\sigma(\overline{X})$. Since $g$ is absolutely symmetric at $\overline{\sigma}$, we know from \cite[Proposition 1]{DSSToh18} that for all $i\in a_l$, $1\leq l\leq r$, the function values $g_i(\overline{\sigma})$ are the same (denoted by $\bar{g}_l$). Therefore, for any $X\in{\cal N}$, we are able to decompose  $G$ into two parts, i.e.,
	\begin{equation}\label{eq:def-GS}
		G_{S}({X}):=\sum_{l=1}^{r}\bar{g}_{l}{\cal U}_{l}(X) \quad {\rm and}\quad G_{R}(X):=G(X)-G_{S}(X) ,
	\end{equation}
	where ${\cal U}_{l}(X):=\sum_{i\in a_{l}}u_{i}v_{i}^\T$ with ${\mathbb O}^{m,n}(X)$.	 It follows from \cite[Lemma 1]{DSSToh18} that there exists an open neighborhood ${\cal B}$ of $\overline{X}$ in ${\cal N}$ such that $G_{S} $ is twice continuously differentiable on ${\cal B}$, and for any $\V^{m\times n}\ni H\to 0$,
	\begin{equation}\label{eq:Gs-diff-formula}
	G_{S}(\overline{X}+H)-G_{S}(\overline{X})=G'_{S}(\overline{X})H+O(\|H\|^{2})	
	\end{equation}
with
		\begin{equation}\label{eq:Gs'-formula}
			G'_{S}(\overline{X})H=\overline{U}\big[\overline{\cal E}^0_1\circ S(\overline{U}^\T H\overline{V}_{1})+\overline{\cal E}^0_2\circ T(\overline{U}^\T H\overline{V}_{1}),
			\quad \overline{\cal F}^0\circ (\overline{U}^\T H\overline{V}_{2})  \big]\overline{V}^\T.
		\end{equation}
In other words, in an open neighborhood of $\overline{X}$, $G_S$ can be regarded as a ``smooth part" of $G$ and $G_R$ can be regarded as the remaining  ``nonsmooth part" of $G$. As we will see in  later developments, this decomposition \eqref{eq:def-GS}  {can simplify} many of our proofs.

	\section{Lipschitz continuity}\label{section:Lip}
	
	In this section,
	we analyze the local Lipschitz continuity of  the spectral operator $G$ defined on a nonempty open set ${\cal N}$. Let  $\overline{X}\in{\cal N}$ be given. Assume that $g$ is locally Lipschitz continuous near $\overline{\sigma}=\sigma(\overline{X})$ with module $L>0$. Therefore, there exists a positive constant $\delta_{0}>0$ such that
	\[
	\|g(\sigma)-g(\sigma')\|\leq L\|\sigma-\sigma'\| \quad \forall\, \sigma,\sigma'\in B(\overline{\sigma},\delta_{0}):=\left\{y\in\hat{\sigma}_{ {\cal N}}\mid \|y-\overline{\sigma}\|\leq \delta_0\right\} .
	\]
	By using the absolutely symmetric property of $g$ on $\hat{\sigma}_{ {\cal N}}$, we obtain the following simple observation.
	
	\begin{proposition}\label{prop:gigj}
	There exist a positive constant $L'>0$ and a positive constant $\delta>0$ such that for any $\sigma\in B(\overline{\sigma},\delta)$,
	\begin{eqnarray}
	|g_{i}(\sigma)-g_{j}(\sigma)|&\leq& L'|\sigma_{i}-\sigma_{j}|\quad \forall\, i,j\in\{1,\ldots, m\},\ i\neq j,\;\; \sigma_{i}\neq \sigma_{j} ,\label{eq:gigj-1}\\[3pt]
	|g_{i}(\sigma)+g_{j}(\sigma)|&\leq& L'|\sigma_{i}+\sigma_{j}|\quad \forall\, i,j\in\{1,\ldots, m\},\;\;  \sigma_{i}+\sigma_{j}>0 ,\label{eq:gigj-2}\\[3pt]
	|g_{i}(\sigma)|&\leq& L'|\sigma_{i}|\quad \forall\, i\in\{1,\ldots,m\},\;\; \sigma_{i}>0\label{eq:gigj-3} .
	\end{eqnarray}
	\end{proposition}
\begin{proof}
	It is easy to check that there exists a positive constant $\delta_{1}>0$ such that for any $\sigma\in B(\overline{\sigma},\delta_{1})$,
	\begin{eqnarray}
	|\sigma_{i}-\sigma_{j}| &\geq &\delta_{1}>0\quad \forall\, i,j\in\{1,\ldots,m\},\ i\neq j,
	\;\; \overline{\sigma}_{i}\neq\overline{\sigma}_{j} ,
	\label{eq:bounded-cond-1}
	\\[3pt]
	|\sigma_{i}+\sigma_{j}|&\geq& \delta_{1}>0\quad \forall\, i,j\in\{1,\ldots,m\},\;\; \overline{\sigma}_{i}+\overline{\sigma}_{j}>0 ,
	\label{eq:bounded-cond-2}
	\\[3pt]
	|\sigma_{i}|&\geq& \delta_{1}>0\quad \forall\, i\in\{1,\ldots,m\},\;\;
	 \overline{\sigma}_{i}>0 .
	\label{eq:bounded-cond-3}
	\end{eqnarray}
	Let $\delta:=\min\{\delta_{0},\delta_{1}\}>0$. Denote $\tau:=\displaystyle{\max_{i,j}}\{|g_{i}(\overline{\sigma})-g_{j}(\overline{\sigma})|, |g_{i}(\overline{\sigma})+g_{j}(\overline{\sigma})|, |g_{i}(\overline{\sigma})|\}\ge 0$, $L_{1}:=(2L\delta+\tau)/\delta$ and $L':=\max\{L_{1},\sqrt{2}L\}$. Let $\sigma$ be any fixed vector in $B(\overline{\sigma},\delta)$.
	
	Firstly, we consider the case that $i,j \in\{1,\ldots,m\}$, $i\neq j$ and $\sigma_{i}\neq\sigma_{j}$. If $\overline{\sigma}_{i}\neq \overline{\sigma}_{j}$, then from (\ref{eq:bounded-cond-1}), we know that
	\begin{eqnarray}
	|g_{i}(\sigma)-g_{j}(\sigma)|&=&|g_{i}(\sigma)-g_{i}(\overline{\sigma})+g_{i}(\overline{\sigma})-g_{j}(\overline{\sigma})+g_{j}(\overline{\sigma})-g_{j}(\sigma)|\nonumber\\
	&\leq&2\|g(\sigma)-g(\overline{\sigma})\|+\tau \leq \frac{2L\delta+\tau}{\delta}|\sigma_{i}-\sigma_{j}|= L_{1}|\sigma_{i}-\sigma_{j}|  .\label{eq:bouded-1-1}
	\end{eqnarray}
	If $\overline{\sigma}_{i}= \overline{\sigma}_{j}$, define $t\in\R^{m}$ by
	\[
	t_{p}:=\left\{ \begin{array}{ll}\sigma_{p}  & \mbox{if $p\neq i,j$,}\\[3pt]
	\sigma_{j}  & \mbox{if $p=i$},\\[3pt]
	\sigma_{i}   & \mbox{if $p=j$},
	\end{array} \right. \quad p=1,\ldots,m .
	\] Then, we have $\|t-\overline{\sigma}\|=\|\sigma-\overline{\sigma}\|\leq \delta$. Moreover, since $g$ is absolutely symmetric on $\hat{\sigma}_{ {\cal N}}$, we have $g_{i}(t)=g_{j}(\sigma)$. Therefore
	\begin{equation}\label{eq:bouded-1-2}
	|g_{i}(\sigma)-g_{j}(\sigma)|=|g_{i}(\sigma)-g_{i}(t)| \leq \|g(\sigma)-g(t)\|\leq L\|\sigma-t\|=\sqrt{2}L|\sigma_{i}-\sigma_{j}| .
	\end{equation}
	Thus, the inequality (\ref{eq:gigj-1}) follows from (\ref{eq:bouded-1-1}) and (\ref{eq:bouded-1-2}) immediately.
	
	Secondly, consider the case $i,j\in\{1,\ldots, m \}$ and $\sigma_{i}+\sigma_{j}>0$. If $\overline{\sigma}_{i}+\overline{\sigma}_{j}>0$, it follows from (\ref{eq:bounded-cond-2}) that
	\begin{eqnarray}
	|g_{i}(\sigma)+g_{j}(\sigma)|&=&|g_{i}(\sigma)-g_{i}(\overline{\sigma})+g_{i}(\overline{\sigma})+g_{j}(\overline{\sigma})-g_{j}(\overline{\sigma})+g_{j}(\sigma)|\nonumber\\
	&\leq& 2\|g(\sigma)-g(\overline{\sigma})\|+\tau \leq\frac{2L\delta+\tau}{\delta}|\sigma_{i}+\sigma_{j}|=L_{1}|\sigma_{i}+\sigma_{j}|  .\label{eq:bouded-2-1}
	\end{eqnarray} If $\overline{\sigma}_{i}+\overline{\sigma}_{j}=0$, i.e., $\overline{\sigma}_{i}=\overline{\sigma}_{j}=0$, define the vector $\hat{t}\in\R^{m}$ by
	\[
	\hat{t}_{p}:=\left\{ \begin{array}{ll}\sigma_{p}  & \mbox{if $p\neq i,j$,}\\
	-\sigma_{j}  & \mbox{if $p=i$},\\
	-\sigma_{i}   & \mbox{if $p=j$},
	\end{array} \right. \quad p=1,\ldots,m .
	\] By noting that $\overline{\sigma}_{i}=\overline{\sigma}_{j}=0$, we obtain that $\|\hat{t}-\overline{\sigma}\|=\|\sigma-\overline{\sigma}\|\leq \delta$.
	Again, since $g$ is absolutely symmetric on $\hat{\sigma}_{ {\cal N}}$,
	we have $g_{i}(\hat{t})=-g_{j}(\sigma)$. Therefore,
	\begin{equation}\label{eq:bouded-2-2}
	|g_{i}(\sigma)+g_{j}(\sigma)|=|g_{i}(\sigma)-g_{i}(\hat{t})|\leq \|g(\sigma)-g(\hat{t})\|\leq L\| \sigma-\hat{t}\|=\sqrt{2}L|\sigma_{i}+\sigma_{j}| .
	\end{equation}
	Thus the inequality (\ref{eq:gigj-2}) follows from  (\ref{eq:bouded-2-1}) and (\ref{eq:bouded-2-2}).
	
	Finally, we consider the case that $i\in\{1,\ldots,m \}$ and $\sigma_{i}>0$. If $\overline{\sigma}_{i}>0$, then we know from (\ref{eq:bounded-cond-3}) that
	\begin{eqnarray}
	|g_{i}(\sigma)|&=&|g_{i}(\sigma)-g_{i}(\overline{\sigma})+g_{i}(\overline{\sigma})|\leq |g_{i}(\sigma)-g_{i}(\overline{\sigma})|+|g_{i}(\overline{\sigma})|\nonumber\\
	&\leq&\|g(\sigma)-g(\overline{\sigma})\|+\tau\leq \frac{2L\delta+\tau}{\delta}|\sigma_{i}|\leq L_{1}|\sigma_{i}| .\label{eq:bouded-3-1}
	\end{eqnarray} If $\overline{\sigma}_{i}=0$, define $s\in\R^{m}$  by
	\[
	s_{p}:=\left\{ \begin{array}{ll}\sigma_{p}  & \mbox{if $p\neq i$,}\\
	0  & \mbox{if $p=i$,}
	\end{array} \right. \quad p=1,\ldots,m .
	\]
	Then, since $\sigma_{i}>0$, we know that $\|s-\overline{\sigma}\|<\|\sigma-\overline{\sigma}\|\leq \delta$. Moreover, since $g$ is absolutely symmetric on $\hat{\sigma}_{ {\cal N}}$, we know that $g_{i}(s)=0$. Therefore, we have
	\begin{equation}\label{eq:bouded-3-2}
	|g_{i}(\sigma)|=|g_{i}(\sigma)-g_{i}(s)|\leq \|g(\sigma)-g(s)\|\leq L\|\sigma-s\|\leq L|\sigma_{i}| .
	\end{equation} Thus,  the inequality (\ref{eq:gigj-1}) follows from (\ref{eq:bouded-3-1}) and (\ref{eq:bouded-3-2}) immediately. This completes the proof.
\end{proof}

	For any fixed $0<\omega\leq\delta_{0}/\sqrt{m}$ and $y\in B(\overline{\sigma},\delta_{0}/(2\sqrt{m})):=\{\|y-\overline{\sigma}\|_{\infty}\leq\delta_{0}/(2\sqrt{m})\}$, the function $g$ is integrable on $V_{\omega}(y):=\{z\in\R^{m}\mid\|y-z\|_{\infty}\leq \omega/2\}$ (in the sense of Lebesgue). Therefore, we know that the function
	\begin{equation}\label{eq:def-Steklov averaged function}
	g(\omega,y):=\frac{1}{\omega ^{m}}\int_{V_{\omega}(y)}g(z)dz
	\end{equation} is well-defined on $(0,\delta_{0}/\sqrt{m}\,]\times B(\overline{\sigma},\delta_{0}/(2\sqrt{m}))$ and is said to be the Steklov averaged function \cite{Steklov1907} of $g$. For the sake of  convenience, we define $g(0,y)=g(y)$. Since $g$ is absolutely symmetric on $\hat{\sigma}_{ {\cal N}}$, it is easy to check that for any fixed $0<\omega\leq\delta_{0}/\sqrt{m}$, the function $g(\omega,\cdot)$ is also absolutely symmetric on $B(\overline{\sigma},\delta_{0}/(2\sqrt{m}))$. It follows from  the definition \eqref{eq:def-Steklov averaged function} that $g(\cdot,\cdot)$ is locally Lipschitz continuous on $(0,\delta_{0}/\sqrt{m}\,]\times B(\overline{\sigma},\delta_{0}/(2\sqrt{m}))$ with the module $L$.  Meanwhile, by elementary calculations, we know that $g(\cdot,\cdot)$ is continuously differentiable on $(0,\delta_{0}/\sqrt{m}\,]\times B(\overline{\sigma},\delta_{0}/(2\sqrt{m}))$ and for any  fixed $\omega\in(0,\delta_{0}/\sqrt{m}\,]$ and $y\in B(\overline{\sigma},\delta_{0}/(2\sqrt{m}))$, $\|g'_{y}(\omega,y)\|\leq L$.
Moreover, it is well known (cf. e.g., \cite[Lemma 1]{Gupal77}) that $g(\omega,\cdot)$ converges to $g$ uniformly on the compact set $B(\overline{\sigma},\delta_{0}/(2\sqrt{m}))$ as $\omega\downarrow 0$. By using the derivative formula of spectral operators  obtained in \cite[(38)]{DSSToh18},  {we can obtain} the following results from \cite[Theorem 4]{DSSToh18} and Proposition \ref{prop:gigj}, directly. For simplicity, we omit the  {detailed} proof here.
	
	\begin{proposition}\label{prop:G'-uni-bounded}
	Suppose that $g$ is locally Lipschitz continuous near $\overline{\sigma}$.   {Let} $g(\cdot,\cdot)$ be the corresponding Steklov averaged function defined in  (\ref{eq:def-Steklov averaged function}). Then, for any given $\omega\in(0,\delta_{0}/\sqrt{m}\,]$, the spectral operator $G(\omega,\cdot)$ with respect to $g(\omega,\cdot)$ is continuously differentiable on $B(\overline{X},\delta_{0}/(2\sqrt{m})):=\{X\in{\cal X}\mid \|\sigma(X)-\overline{\sigma}\|_{\infty}\leq\delta_{0}/(2\sqrt{m})\}$, and there exist two positive constants $\delta_{1}>0$ and $\overline{L}>0$ such that
	\begin{equation}\label{eq:G'-uni-bounded}
	\|G'(\omega,X)\|\leq \overline{L}\quad \forall\, 0<\omega\leq \min\{\delta_{0}/\sqrt{m},\delta_{1}\}\ {\rm and}\ X\in B(\overline{X},\delta_{0}/(2\sqrt{m})) .
	\end{equation} Moreover, $G(\omega,\cdot)$ converges to $G$ uniformly in the compact set $B(\overline{X},\delta_{0}/(2\sqrt{m}))$ as $\omega\downarrow 0$.
	 \end{proposition}
	
	Proposition \ref{prop:G'-uni-bounded} allows us to derive the following result on the local Lipschitz continuity of  spectral operators.
	
	\begin{theorem}\label{thm:Lip-spectral-op}
	Suppose that $\overline{X}$ has the SVD (\ref{eq:Y-eig-Z-SVD}).  The spectral operator $G$ is locally Lipschitz continuous near $\overline{X}$ if and only if $g$ is locally Lipschitz continuous near $\overline{\sigma}=\sigma(\overline{X})$.
	\end{theorem}
	\begin{proof}
		$``\Longleftarrow"$ Suppose that $g$ is locally Lipschitz continuous near $\overline{\sigma}=\sigma(\overline{X})$ with module $L>0$, i.e., there exists a positive constant $\delta_{0}>0$ such that
	\[
	\|g(\sigma)-g(\sigma')\|\leq L\|\sigma-\sigma'\| \quad \forall\, \sigma,\sigma'\in B(\overline{\sigma},\delta_{0}) .
	\] By Proposition \ref{prop:G'-uni-bounded}, for any  $\omega\in(0,\delta_{0}/\sqrt{m}\,]$, the spectral operator $G(\omega,\cdot)$ defined with respect to the Steklov averaged function $g(\omega, \cdot)$ is continuously differentiable. Since $G(\omega,\cdot)$ converges to $G$ uniformly in the compact set $B(\overline{X},\delta_{0}/(2\sqrt{m}))$ as $\omega\downarrow 0$, we know that for any $\varepsilon>0$, there exists a constant $\delta_{2}>0$ such that for any $0<\omega\leq\delta_{2}$,
	\[
	\|G(\omega,X)-G(X)\|\leq \varepsilon\quad \forall\,X\in B(\overline{X},\delta_{0}/(2\sqrt{m})) .
	\] Fix any $X,X'\in B(\overline{X},\delta_{0}/(2\sqrt{m}))$ with $X\neq X'$. By Proposition \ref{prop:G'-uni-bounded}, we know that there exists $\delta_{1}>0$ such that (\ref{eq:G'-uni-bounded}) holds. Let $\bar{\delta}:=\min\{\delta_{1}, \delta_{2},\delta_{0}/\sqrt{m}\}$. Then, by the mean value theorem, we know that
	\begin{eqnarray*}
	\|G(X)-G(X')\|&=&\|G(X)-G(\omega,X)+G(\omega,X)-G(\omega,X')+G(\omega,X')-G(X')\|\\
	&\leq&2\varepsilon+\|\int_{0}^{1}G'(\omega,X+t(X-X'))dt\|\leq\overline{L}\|X-X'\|+2\varepsilon\quad \forall\, 0<\omega<\bar{\delta} .
	\end{eqnarray*} Since $X,X'\in B(\overline{X},\delta_{0}/(2\sqrt{m}))$ and $\varepsilon>0$ are arbitrary, by letting $\varepsilon\downarrow 0$, we obtain that
	\[
	\|G(X)-G(X')\|\leq\overline{L}\|X-X'\|\quad \forall\,X,X'\in B(\overline{X},\delta_{0}/(2\sqrt{m})) .
	\] Thus $G$ is locally Lipschitz continuous near $\overline{X}$.
	
	$``\Longrightarrow"$ Suppose that $G$ is locally Lipschitz continuous near $\overline{X}$ with module $L>0$, i.e., there exists an open neighborhood ${\cal B}$ of $\overline{X}$ in ${\cal N}$ such that for any $X,X'\in{\cal B}$,
	\[
	\|G(X)-G(X')\|\leq L\|X-X'\| .
	\]
	Let $(\overline{U},\overline{V})\in{\mathbb O}^{m\times n}(\overline{X})$ be fixed. For any $y\in\hat{\sigma}_{\cal N}$, we   define $Y:=\overline{U}\left[{\rm Diag}(y) \quad 0 \right]\overline{V}^\T$. Then, we know from \cite[Proposition 3]{DSSToh18} that $G(Y)=\overline{U}\left[{\rm Diag}(g(y)) \quad 0 \right]\overline{V}^\T$. Therefore, we obtain that there exists  an open neighborhood ${\cal B}_{\overline{\sigma}}$ of $\overline{\sigma}$ in $\hat{\sigma}_{\cal N}$ such that
	\[
	\|g(y)-g(y')\|=\|G(Y)-G(Y')\|\leq L\|Y-Y'\|=L\|y-y'\|\quad \forall\, y,y'\in{\cal B}_{\overline{\sigma}} .
	\]
This completes the proof.
	\end{proof}

	
\section{Bouligand-differentiability}\label{section:B-diff}

In this section, we shall study the $\rho$-order Bouligand-differentiability of spectral operators with $0<\rho\leq 1$, which is a slightly  stronger property than the directional differentiability studied in \cite[Theorem 3]{DSSToh18}.

Let ${\cal Z}$ be a finite dimensional real Euclidean space equipped with an inner product $\langle \cdot,\cdot \rangle$ and its induced norm $\|\cdot\|$. Let ${\cal O}$ be an open set in ${\cal Z}$ and ${\cal Z}'$ be another finite dimensional real Euclidean space. The function $F: {\cal O}\subseteq{\cal Z}  \to  {\cal Z}^\prime$ is said to be  {\it B(ouligand)-differentiable} \cite{Robinson87} (see also \cite{Pang90,FPang03,Pang91} for more details) at $z\in{\cal O}$ if for any $h\in{\cal Z}$ with $h\to 0$,
\begin{equation*}\label{eq:def-B-diff}
F(z+h)-F(z)-F'(z;h)=o(\|h\|) .
\end{equation*}
It is well known (cf. \cite{Shapiro90}) that if $F$ is locally Lipschitz continuous then $F$ is B-differentiable at $z\in{\cal O}$ if and only if $F$ is directionally differentiable at $z$. If the spectral operator $G$ is directionally differentiable, then the corresponding directional derivative formula  is presented in \cite[(21) in Theorem 3]{DSSToh18}. More precisely, since $g$ is absolutely symmetric on the nonempty open set $\hat{\sigma}_{ {\cal N}}$, it is easy to see that the directional derivative $\phi:=g'(\overline{\sigma};\cdot):\R^{m}\to\R^{m}$ satisfies
	\begin{equation}\label{eq:dir-diff-symmetric}
		g'(\overline{\sigma};Qh)=Qg'(\overline{\sigma};h)\quad \forall\, Q\in\pm\mathbb{P}_{\overline{\sigma}}^m\quad {\rm and} \quad \forall\,h\in\R^{m}\,,
	\end{equation} where $\pm\mathbb{P}_{\overline{\sigma}}^m$ is the subset defined with respect to $\overline{\sigma}$ by $\pm\mathbb{P}_{\overline{\sigma}}^m:=\left\{{Q}\in\pm\mathbb{P}^m\,|\, \overline{\sigma}={Q}\overline{\sigma}\right\}$. Thus, we know that the function $\phi$ is a mixed symmetric mapping, with respect to
$ \mathbb{P}^{|a_1|}\times\ldots\times\mathbb{P}^{|a_r|}\times\pm\mathbb{P}^{|b|}$,
over ${\cal V}:=\R^{|a_1|}\times\ldots\times\R^{|a_r|}\times\R^{|b|}$. Let $\Psi:=G'(\overline{X};\cdot):\V^{m\times n}\to\V^{m\times n}$ be the directional derivative of $G$ at $\overline{X}$. Let ${\cal W}:=\S^{|a_{1}|}\times\ldots\times\S^{|a_{r}|}\times\V^{|b|\times(n-|a|)}$. We know from \cite[(21) Theorem 3]{DSSToh18} that for any $H\in\V^{m\times n}$,
\begin{eqnarray}
\Psi(H) &=&G'(\overline{X};H)= \overline{U}\left[\overline{\cal E}^0_1\circ S(\overline{U}^\T H\overline{V}_{1})+\overline{\cal E}^0_2\circ T(\overline{U}^\T H\overline{V}_{1}) \quad \overline{\cal F}^0\circ \overline{U}^\T H\overline{V}_{2}  \right]\overline{V}^\T  +\overline{U} {\widehat \Phi}(D(H))\overline{V}^\T \nonumber
\\[3pt]
&=&  {G_S^\prime(\overline{X};H) + \overline{U} {\widehat \Phi}(D(H))\overline{V}^\T},
\label{eq:def-Psi}
\end{eqnarray}
where  $D(H)=\left(S(\widetilde{H}_{a_1a_1}),\ldots,S(\widetilde{H}_{a_ra_r}),\widetilde{H}_{b\bar{a}}\right)\in{\cal W}$, $\widetilde{H}=\overline{U}^\T H\overline{V}$, ${\Phi}:{\cal W}\to{\cal W}$ being the spectral operator defined with respect to the mixed symmetric mapping $\phi=g'(\overline{\sigma};\cdot)$,
and ${\widehat \Phi}:{\cal W}\to \V^{m\times n}$ is  {defined by}
	\begin{equation}\label{eq:def-Diag-Phi}
		{\widehat \Phi}(W):= \left[\begin{array}{cc} {\rm Diag}\left(\Phi_{1}(W), \dots, \Phi_{r}(W)\right)  & 0
			\\[2mm]
			0 &  \Phi_{r+1}(W)
		\end{array}\right] \quad { \forall\; W\in{\cal W}.}
	\end{equation}

A stronger notion than B-differentiability is $\rho$-order B-differentiability with $\rho>0$. The function $F: {\cal O}\subseteq{\cal Z}  \to  {\cal Z}^\prime$ is said to be  {\it $\rho$-order B-differentiable} at $z\in{\cal O}$ if for any $h\in{\cal Z}$ with $h\to 0$,
\begin{equation*}\label{eq:def-rho-order-B-diff}
F(z+h)-F(z)-F'(z;h)=O(\|h\|^{1+\rho}) .
\end{equation*}
Let $\overline{X}\in\V^{m\times n}$ be given. We have the following results on the $\rho$-order B-differentiability of spectral operators.
\begin{theorem}\label{thm:rho-order-B-diff-spectral-op}
	Suppose that $\overline{X}\in{\cal N}$ has the SVD (\ref{eq:Y-eig-Z-SVD}).  Let $0<\rho\leq 1$ be given.
	\begin{itemize}
		\item[(i)]  {If} $g$ is locally Lipschitz continuous near $\sigma(\overline{X})$ and $\rho$-order B-differentiable at $\sigma(\overline{X})$, then $G$ is $\rho$-order B-differentiable at $\overline{X}$.
		\item[(ii)]  {If} $G$ is $\rho$-order B-differentiable at $\overline{X}$, then $g$ is $\rho$-order B-differentiable at $\sigma(\overline{X})$.
	\end{itemize}
\end{theorem}
\begin{proof}
Without loss of generality, we only prove the results for the case that $\rho =1$.

(i) For any $H\in\V^{m\times n}$, denote $X=\overline{X}+H$. Let $U\in{\mathbb O}^{m}$ and $V\in{\mathbb O}^{n}$ be such that
\begin{equation}\label{eq:def-Y-Z-B-diff}
X=U[\Sigma(X) \quad 0]V^\T  .
\end{equation}
Denote $\sigma=\sigma(X)$. Let $G_{S}(X)$ and $G_{R}(X)$ be defined by (\ref{eq:def-GS}). Therefore, by \eqref{eq:Gs-diff-formula}, we know that for any $H\to{0}$,
\begin{equation}\label{eq:smoothpart-B-diff-Spectral-op}
G_{S}(X)-G_{S}(\overline{X})=G'_{S}(\overline{X})H+O(\|H\|^{2}),
\end{equation}
where $G_S'(\overline{X})H$ is given by \eqref{eq:Gs'-formula}. For $H\in\V^{m\times n}$ sufficiently small, we have ${\cal U}_{l}(X)={\sum_{i\in a_{l}}}u_{i}v_{i}^\T $, $l=1,\ldots,r$. Therefore,  we know that
\begin{equation}\label{eq:GR-def-B-diff-spectral-op}
G_{R}(X)=G(X)-G_{S}(X)=\sum_{l=1}^{r+1}\Delta_{l}(H) ,
\end{equation}
where
\[
\Delta_{l}(H)=\sum_{i\in a_{l}}(g_{i}(\sigma)-g_i(\overline{\sigma}))u_{i}v_{i}^\T  \quad l=1,\ldots,r \quad {\rm and} \quad \Delta_{r+1}(H)=\sum_{i\in b}g_i(\sigma)u_{i}v_{i}^\T.
\]

\noindent {(a)}
We first consider the case that $\overline{X}=[\Sigma(\overline{X}) \quad 0 ]$. Then,  we know from the directional differentiability of single values (cf. e.g., \cite[Theorem 7]{Lancaster64}, \cite[Proposition 1.4]{Torki01} and \cite[Section 5.1]{LSendov05b}) that for any  $H$ sufficiently small,
\begin{equation}\label{eq:eig-singular-B-diff-1}
\sigma=\overline{\sigma}+\sigma'(\overline{X};H)+O(\|H\|^{2}) ,
\end{equation}
where $\sigma'(\overline{X};H)=\left(\lambda(S(H_{a_{1}a_{1}})),\ldots,\lambda(S(H_{a_{r}a_{r}})),\sigma([H_{bb}\quad H_{bc}])\right)\in\R^{m}$. Denote $h:=\sigma'(\overline{X};H)$. Since $g$ is locally Lipschitz continuous near $\overline{\sigma}$ and $1$-order B-differentiable at $\overline{\sigma}$, we know that for any $H$ sufficiently small,
\begin{equation*}
g(\sigma)-g(\overline{\sigma})=g( {\overline{\sigma}}+h+O(\|H\|^{2}))-g(\overline{\sigma})
= g( {\overline{\sigma}}+h)-g(\overline{\sigma})+O(\|H\|^{2})=g'(\overline{\sigma};h)+O(\|H\|^{2}) .
\end{equation*}
Let $\phi=g'(\overline{\sigma};\cdot)$. Since $u_{i}v_{i}^\T $, $i=1,\ldots,m$ are uniformly bounded, we obtain that for $H$ sufficiently small,
\begin{eqnarray*}
	\Delta_{l}(H) &=& U_{a_{l}}{\rm Diag}(\phi_{l}(h))V_{a_{l}}^\T +O(\|H\|^{2}),\quad  l=1,\ldots,r ,
	\\
	\Delta_{r+1}(H)&=& U_{b}{\rm Diag}(\phi_{r+1}(h))V_{b}^\T +O(\|H\|^{2}) .
\end{eqnarray*}
Again, we know from \cite[Proposition 7]{DSToh10} that there exist $Q_{l}\in{\mathbb O}^{|a_{l}|}$, $M\in{\mathbb O}^{|b|}$ and $N=[N_{1}\quad N_{2}]\in{\mathbb O}^{n-|a|}$ with $N_{1}\in\V^{(n-|a|)\times |b|}$ and $N_{2}\in\V^{(n-|a|)\times(n-m)}$ (depending on $H$) such that
\begin{eqnarray*}
	U_{a_{l}} &=&\left[\begin{array}{c} O(\|H\|)  \\[3pt] Q_{l}+O(\|H\|) \\[3pt] O(\|H\|) \end{array}\right], \quad  V_{a_{l}}=\left[\begin{array}{c} O(\|H\|)  \\[3pt] Q_{l}+O(\|H\|) \\[3pt] O(\|H\|) \end{array}\right], \ l=1,\ldots,r ,
	\\[3pt]
	U_{b}&=&\left[\begin{array}{c} O(\|H\|) \\[3pt] M+O(\|H\|)\end{array}\right],\quad  [V_{b}\quad V_{c}]=\left[\begin{array}{c} O(\|H\|) \\[3pt]
		N+O(\|H\|)\end{array}\right]   .
\end{eqnarray*}
Since $g$ is locally Lipschitz continuous near $\overline{\sigma}$ and directionally differentiable at $\overline{\sigma}$, we know from \cite[Theorem A.2]{Robinson87} or \cite[Lemma 2.2]{QSun93} that the directional derivative $\phi$ is globally Lipschitz continuous on $\R^{m}$. Thus, for $H$ sufficiently small, we have $\|\phi(h)\|=O(\|H\|)$. Therefore,  we obtain that
\begin{eqnarray}
\Delta_{l}(H) &=& \left[\begin{array}{ccc} 0 & 0  & 0  \\[3pt] 0 & Q_{l}{\rm Diag}(\phi_{l}(h))Q_{l}^\T   & 0   \\[3pt] 0 & 0  & 0 \end{array}\right]
+O(\|H\|^{2}),\quad l=1,\ldots,r ,
\label{eq:Delta-k-B-diff-1}
\\[6pt]
\Delta_{r+1}(H)&=&\left[\begin{array}{cc}0 & 0  \\[3pt]0 & M{\rm Diag}(\phi_{r+1}(h))N_{1}^\T  \end{array}\right]+O(\|H\|^{2}) .
\label{eq:Delta-k-B-diff-2}
\end{eqnarray}
Again, it follows from \cite[Proposition 7]{DSToh10} that
\begin{eqnarray}
S(H_{a_{l}a_{l}}) &=& Q_{l}(\Sigma(X)_{a_{l}a_{l}}-\overline{\nu}_{l}I_{|a_{l}|})Q_{l}^\T +O(\|H\|^{2}),\quad l=1,\ldots,r ,
\label{eq:H-H'-2}
\\[3pt]
[H_{bb}\quad H_{bc}] &=& M(\Sigma(X)_{bb}-\overline{\nu}_{r+1}I_{|b|})N_{1}^\T +O(\|H\|^{2}) .
\label{eq:H-H'-3}
\end{eqnarray}
Since $g$ is locally Lipschitz continuous near $\overline{\sigma}=\sigma(\overline{X})$,  we know from Theorem \ref{thm:Lip-spectral-op} that the spectral operator $G$ is locally Lipschitz continuous near $\overline{X}$. Therefore, we know from \cite[Theorem 3 and Remark 1]{DSSToh18} that $G$ is  {directionally differentiable} at $\overline{X}$. Thus, from  \cite[Theorem A.2]{Robinson87} or \cite[Lemma 2.2]{QSun93}, we know that $G'(\overline{X},\cdot)$ is globally Lipschitz continuous on $\V^{m\times n}$. Moreover, from the definition of directional derivative and the absolutely symmetry of $g$ on the nonempty open set $\hat{\sigma}_{ {\cal N}}$, it is easy to see that the directional derivative $\phi:=g'(\overline{\sigma};\cdot)$ is actually a mixed symmetric mapping over the space ${\cal V}:=\R^{|a_1|}\times\ldots\times\R^{|a_r|}\times\R^{|b|}$.
Let ${\cal W}:=\S^{|a_{1}|}\times\ldots\times\S^{|a_{r}|}\times\V^{|b|\times(n-|a|)}$. Thus, the corresponding spectral operator $\Phi$ defined with respect to $\phi$ is globally Lipschitz continuous on the space ${\cal W}$. Hence, we know from (\ref{eq:GR-def-B-diff-spectral-op}) that for $H$ sufficiently small,
\begin{equation}\label{eq:nonsmoothpart-B-diff-Spectral-op-diag}
G_{R}(X)= {\widehat
	\Phi}(D(H))+O(\|H\|^{2})   ,
\end{equation} where
$D(H)=\left(S(H_{a_{1}a_{1}}),\ldots,S(H_{a_{r}a_{r}}), H_{b\bar{a}} \right)\in{\cal W}$ and ${\widehat
\Phi}$ is defined by (\ref{eq:def-Diag-Phi})

\bigskip
\noindent {(b)}
Next, consider the general case that $\overline{X}\in\V^{m\times n}$. For any $H\in\V^{m\times n}$,
 {we rewrite (\ref{eq:def-Y-Z-B-diff}) by using the singular value decomposition of $\overline{X}$ as follows:} $ {\widetilde{X} :=} [\Sigma(\overline{X})\quad 0 ]+\overline{U}^\T H\overline{V}=\overline{U}^\T U[\Sigma( {X}) \quad 0]V^\T \overline{V}$.
 Then, since $ \overline{U}$ and $ \overline{V}$ are unitary matrices, we know from (\ref{eq:nonsmoothpart-B-diff-Spectral-op-diag}) that
\begin{equation}\label{eq:nonsmoothpart-B-diff-Spectral-op}
G_{R}(X)=  {\overline{U} G_R(\widetilde{X})\overline{V}^\T= }
\overline{U} {\widehat
	\Phi}(D(H)) \overline{V}^\T  +O(\|H\|^{2}) ,
\end{equation}
where $D(H)=\left(S(\widetilde{H}_{a_{1}a_{1}}),\ldots,S(\widetilde{H}_{a_{r}a_{r}}), \widetilde{H}_{b\bar{a}} \right)$ and $\widetilde{H}=\overline{U}^\T H\overline{V}$. Thus, by combining
 {(\ref{eq:def-Psi})}, (\ref{eq:smoothpart-B-diff-Spectral-op}) and (\ref{eq:nonsmoothpart-B-diff-Spectral-op}) and noting that
$ {G(\overline{X})}=G_{S}(\overline{X})$, we obtain that for any  $H\in\V^{m\times n}$ sufficiently close to $0$,
\begin{eqnarray*}
&& G(X)-G(\overline{X})-G'(\overline{X};H)
\\[3pt]
&=&  {G_R(X) + G_S(X) - G_S(\overline{X})-G'(\overline{X};H)
= G_R(X) -\overline{U} {\widehat
	\Phi}(D(H)) \overline{V}^\T+ O(\|H\|^{2})
}
=O(\|H\|^{2}) ,
\end{eqnarray*}
where the directional derivative $G'(\overline{X};H)$ of $G$ at $\overline{X}$ along $H$ is given by \eqref{eq:def-Psi}. This implies that $G$ is $1$-order B-differentiable at $\overline{X}$.

(ii) Suppose that $G$ is $1$-order B-differentiable at $\overline{X}$. Let $(\overline{U},\overline{V})\in{\mathbb O}^{m\times n}(\overline{X})$ be fixed.  For any $h\in\R^{m}$, let $H=\overline{U}[{\rm Diag}(h)\quad 0]\overline{V}^\T \in\V^{m\times n}$. We know from \cite[Proposition 3]{DSSToh18} that for all $h$ sufficiently close to $0$, $G(\overline{X}+H)=\overline{U}{\rm Diag}(g(\overline{\sigma}+h))\overline{V}_1^\T$. Therefore, we know from the assumption that
\[
{\rm Diag}(g(\overline{\sigma}+h)-g(\overline{\sigma}))= \overline{U}^\T \left(G(\overline{X}+H)-G(\overline{X})\right)\overline{V}_1=\overline{U}^\T G'(\overline{X};H)\overline{V}_1+O(\|H\|^{2}) .
\]
This shows that $g$ is $1$-order B-differentiable at $\overline{\sigma}$. The proof  is completed.
\end{proof}

\section{G-semismoothness}\label{section:semismoothness}

Let ${\cal Z}$ and ${\cal Z}'$ be two finite dimensional real Euclidean spaces and ${\cal O}$ be an open set in ${\cal Z}$. Suppose that $F: {\cal O}  \subseteq {\cal Z}  \to  {\cal Z}^\prime$ is a locally Lipschitz continuous function on ${\cal O}$. Then, according to Rademacher's theorem, $F$ is almost everywhere differentiable (in the sense of Fr\'{e}chet)  in ${\cal O}$.  Let ${\cal D}_F$ be the set of points in ${\cal O}$ where $F$ is differentiable. Let $F^\prime(z)$ be the derivative of $F$ at $z\in {\cal D}_F$. Then the  {\it B(ouligand)-subdifferential} of $F$ at $z\in {\cal O}$ is denoted by \cite{Qi93}:
\[
\partial_{B}F(z):=\left\{ \lim_{{\cal D}_{F}\ni z^{k}\to z} F'(z^{k}) \right\}
\] and  the  {\it Clarke generalized Jacobian} of $F$ at $z\in {\cal O}$ \cite{Clarke83}  takes the form:
\[
\partial F(z)={\rm conv}\{\partial_{B} F(z)\} ,
\] where ``conv'' stands for the convex hull in the usual sense of convex analysis \cite{Rockafellar70}. The function $F$ is said to be G-semismooth  at a point $z \in {\cal O}$  if for any $y\to z $ and $V \in \partial F(y)$,
\[
F(y) -F(z) - V(y-z) = o(\|y- z\| )  .
\]
A stronger notion than G-semismoothness is $\rho$-order G-semismoothness with $\rho>0$. The function $F$ is said to be  $\rho$-order G-semismooth at $z$  if
for any $y\to z $ and $V \in \partial F(y)$,
\[
F(y) - F(z) - V(y- z) = O(\|y-z\|^{1+\rho})  .
\]
In particular, the function $F$ is said to be strongly G-semismooth at $z$ if $F$ is $1$-order G-semismooth at $z$. Furthermore, the  function $F$ is said to be  ($\rho$-order, strongly) semismooth at $z\in {\cal O}$  if (i)
the directional derivative of $F$ at $z$ along any direction $d\in {\cal Z}$, denoted by $F^\prime(z;d)$,  exists; and (ii) $F$ is   ($\rho$-order, strongly) G-semismooth.

The following result taken from \cite[Theorem 3.7]{SSun02} provides a convenient tool for proving the G-semismoothness of Lipschitz functions.
\begin{lemma}\label{lem:semismoothness-equiv}
	Let $F: {\cal O}  \subseteq {\cal Z} \to {\cal Z}^\prime $ be a locally Lipschitz continuous function on the open set ${\cal O}$,  {and  $\rho >0$ be} a constant. $F$ is $\rho $-order G-semismooth (G-semismooth) at $z$ if and only if  for any ${\cal D}_F\ni y\to  z $,
	\begin{equation}\label{eq:G-semismoothness-equiv}
	F(y) - F(z) - F^\prime(y)(y- z) = O(\|y- z\|^{1+\rho}) \quad \big(=o(\|y- z\|)\big)   .
	\end{equation}
\end{lemma}

Let  $\overline{X}\in{\cal N}$ be given. Assume that $g$ is locally Lipschitz continuous near $\overline{\sigma}=\sigma(\overline{X})$.  {Then} from Theorem \ref{thm:Lip-spectral-op} we know that the corresponding spectral operator $G$ is locally Lipschitz continuous near $\overline{X}$. The following theorem is on the G-semismoothness of the spectral operator $G$.

\begin{theorem}\label{thm:rho-order-G-semismooth-spectral-op}
	Suppose that $\overline{X}\in{\cal N}$ has the  {singular value} decomposition (\ref{eq:Y-eig-Z-SVD}).  Let $0<\rho\leq 1$ be given. $G$ is $\rho$-order G-semismooth at $\overline{X}$ if and only if $g$ is $\rho$-order G-semismooth at $\overline{\sigma}$.
\end{theorem}
\begin{proof}
	Without loss of generality, we only prove the result for the case that $\rho=1$.

$``\Longleftarrow"$ For any $H\in\V^{m\times n}$, denote $X=\overline{X}+H$. Let $U\in{\mathbb O}^{m}$ and $V\in{\mathbb O}^{n}$ be such that
\begin{equation}\label{eq:def-Y-Z-semismooth}
X=U[\Sigma(X)\quad 0]V^\T.
\end{equation}
Denote $\sigma=\sigma(X)$.  {Recall the mappings $G_{S}$ and $G_{R}$} defined in (\ref{eq:def-GS}). We know from \cite[Proposition 8]{DSToh10} that there exists an open neighborhood ${\cal B}\subseteq{\cal N}$ of $\overline{X}$ such that $G_{S}$ twice continuously differentiable on ${\cal B}$ and
\begin{eqnarray}
&& G_{S}(X)-G_{S}(\overline{X})=\sum_{l=1}^{r}\bar{g}_{l}\,{\cal U}'_{l}(X)\,H
+O(\|H\|^{2}) \nonumber\\
&=&\sum_{l=1}^{r} \bar{g}_{l}\left\{U[\Gamma_{l}(X)\circ S(U^\T HV_{1} )+\Xi_{l}(X)\circ
T(U^\T HV_{1})]V_{1}^\T +U(\Upsilon_{l}(X)\circ U^\T HV_{2} )V_{2}^\T \right\} +O(\|H\|^{2}) ,
\label{eq:smoothpart-semismooth-Spectral-op}
\end{eqnarray}
where for each $l\in\{1,\ldots,r\}$, $\Gamma_{l}(X)$, $\Xi_{l}(X)$  and $\Upsilon_{l}(X)$ are given by
\cite[(40)--(42)]{DSToh10}, respectively. By taking a smaller ${\cal B}$ if necessary, we
 {may} assume that for any $X\in{\cal B}$ and $l,l'\in\{1,\ldots,r\}$,
\begin{equation}\label{eq:sigma_X-diff}
\sigma_i(X)>0,\quad \sigma_i(X)\neq \sigma_j(X)\quad \forall\, i\in a_l,\ j\in a_{l'}\ {\rm and}\ l\neq l' .
\end{equation} Since $g$ is locally Lipschitz continuous near $\overline{\sigma}$, we know that for any $H$ sufficiently small,
\begin{equation}\label{eq:g-Lip-semismooth}
\bar{g}_{l}=g_{i}(\sigma)+O(\|H\|)\quad \forall\, i\in a_l,\quad l=1,\ldots,r .
\end{equation}
By noting that $U\in{\mathbb O}^{m}$ and $V\in{\mathbb O}^{n}$ are uniformly bounded, we know from \eqref{eq:smoothpart-semismooth-Spectral-op} and \eqref{eq:g-Lip-semismooth} that for any $X\in{\cal B}$ (shrinking ${\cal B}$ if necessary),
\begin{equation}\label{eq:GX-GbarX-semismooth}
G_{S}(X)-G_{S}(\overline{X})=U\left[{\cal E}^0_{1}\circ S(U^\T HV_{1} ) + {\cal E}^0_{2}\circ
T(U^\T HV_{1} ) \quad {\cal F}^0\circ U^\T HV_{2} \right]V^\T +O(\|H\|^{2}) ,
\end{equation}
where ${\cal E}^0_{1}$, ${\cal E}^0_{2}$ and ${\cal F}^0$ are the corresponding real matrices defined in \eqref{eq:def-matric-E1}--\eqref{eq:def-matric-F} (depending on $X$), respectively.

Let $X\in{\cal D}_{G}\cap{\cal B}$, where ${\cal D}_{G}$ is the set of points in $\V^{m\times n}$ for which $G$ is (F-)differentiable. Define the corresponding index sets in $\{1,\ldots,m\}$ for $X$ by $a':=\{i\mid \sigma_i(X)>0\}$ and $b':=\{i\mid \sigma_i(X)=0\}$. By \eqref{eq:sigma_X-diff}, we have
\begin{equation}
a'\supseteq a \quad {\rm and} \quad  b'\subseteq b .
\end{equation}
We know from  \cite[Theorem 4]{DSSToh18} that
\begin{equation}\label{eq:diff-Y-spectral-op-semismooth}
G'(X)H= U[{\cal E}_{1}\circ S(U^\T HV_1)+{\cal E}_{2}\circ T(U^\T HV_1)+ {\rm Diag}\left({\cal C}{\rm diag}(S(U^\T HV_1))\right)
\quad {\cal F}\circ U^\T HV_2 ]V^\T   ,
\end{equation}
where $\eta$, ${\cal E}_{1}$, ${\cal E}_{2}$, ${\cal F}$ and  $\cal C$ are defined by \cite[(33)--(36)]{DSSToh18} with respect to $\sigma$, respectively.  Denote $\Delta(H):=G'(X)H-(G_{S}(X)-G_{S}(\overline{X}))$.
Moreover, since there exists an integer $j\in\{0,\ldots,|b|\}$ such that $|a'|=|a|+j$, we can define two index sets $b_1:=\{|a|+1,\ldots,|a|+j\}$ and $b_2:=\{|a|+j+1,\ldots,|a|+|b|\}$ such that  $a'=a\cup b_1$ and $b'=b_2$.
From (\ref{eq:GX-GbarX-semismooth}) and (\ref{eq:diff-Y-spectral-op-semismooth}), we obtain that
\begin{equation}\label{eq:Delta2-semismooth-spectral-op}
\Delta(H)=U
{\widehat R} (H)
V^\T +O(\|H\|^{2}) ,
\end{equation}
where  ${\widehat R} (H)\in \V^{m\times n}$ is defined by
\[
{\widehat R}(H):= \left[\begin{array}{cc} {\rm Diag}\left( R_{1}(H), \dots, R_{r}(H) \right)  & 0
\\
0 &  R_{r+1}(H)
\end{array}\right],
\]
\begin{eqnarray}\label{eq:R-k-semismooth-spectral-op-2}
R_{l}(H) &=& ({\cal E}_{1})_{a_{l}a_{l}}\circ S(U_{a_{l}}^\T HV_{a_{l}})+{\rm Diag}\left(({\cal C}{\rm diag}(S(U^\T HV_1)))_{a_la_l}\right), \  l=1,\ldots,r,
\\
\label{eq:R-k-semismooth-spectral-op-3}
R_{r+1}(H)&=&\left[\begin{array}{ccc}
({\cal E}_{1})_{b_1b_1}\circ S(U_{b_1}^\T HV_{b_1})+{\rm Diag}\left(({\cal C}{\rm diag}(S(U^\T HV_1)))_{b_1b_1}\right)  & 0 & 0 \\[3pt]
0 & \gamma U_{b_2}^\T HV_{b_2}  &\; \gamma U_{b_2}^\T HV_{2}
\end{array}\right]
\quad
\end{eqnarray}
and $\gamma:=(g'(\sigma))_{ii}$ for any $i\in b_2$.
By (\ref{eq:Y-eig-Z-SVD}),  we obtain from (\ref{eq:def-Y-Z-semismooth}) that
\[
\left[ \Sigma(\overline{X}) \quad 0\right]+\overline{U}^\T H\overline{V}=\overline{U}^\T U\left[ \Sigma(X)\quad 0 \right]V^\T \overline{V} .
\] Let $\widehat{H}:=\overline{U}^\T H\overline{V}$, $\widehat{U}:=\overline{U}^\T U$ and $\widehat{V}:=\overline{V}^\T V$. Then, $U^\T HV=\widehat{U}^\T \overline{U}^\T H\overline{V}\widehat{V}=\widehat{U}^\T \widehat{H}\widehat{V}$. We know from \cite[(31) in Proposition 7]{DSToh10} that there exist $Q_{l}\in{\mathbb O}^{|a_{l}|}$, $l=1,\ldots,r$ and $M\in{\mathbb O}^{|b|}$, $N\in{\mathbb O}^{n-|a|}$  such that
\begin{eqnarray*}
	&U_{a_{l}}^\T HV_{a_{l}}=\widehat{U}_{a_{l}}^\T  \widehat{H} \widehat{V}_{a_{l}}=Q_{l}^\T  \widehat{H}_{a_{l}a_{l}}Q_{l}+O(\|H\|^{2}),\quad l=1,\ldots,r ,&
	\\[3pt]
	&\left[ U_{b}^\T HV_{b}\quad U_{b}^\T HV_{2} \right]
	= \left[ \widehat{U}_{b}^\T  \widehat{H} \widehat{V}_{b}\quad \widehat{U}_{b}^\T  \widehat{H} \widehat{V}_{2} \right]=M^\T \left[\widehat{H}_{bb}\quad \widehat{H}_{bc}\right]N+O(\|H\|^{2}) .&
\end{eqnarray*}
Moreover, from \cite[(32) and (33) in Proposition 7]{DSToh10}, we obtain that
\begin{eqnarray*}
	&S(U_{a_{l}}^\T HV_{a_{l}})=Q_{l}^\T  S(\widehat{H}_{a_{l}a_{l}})Q_{l}+O(\|H\|^{2})=\Sigma(X)_{a_{l}a_{l}}-\Sigma(\overline{X})_{a_{l}a_{l}}+O(\|H\|^{2}),\quad l=1,\ldots,r ,&
	\\[3pt]
	&\left[ U_{b}^\T HV_{b}\quad U_{b}^\T HV_{2} \right]=M^\T \left[\widehat{H}_{bb}\quad \widehat{H}_{bc}\right]N=\left[\Sigma(X)_{bb}-\Sigma(\overline{X})_{bb} \quad 0\right]+O(\|H\|^{2}) .&
\end{eqnarray*}
Denote $h=\sigma'(X;H)\in\R^{m}$. Since the  {singular} value functions are strongly semismooth \cite{SSun03}, we know that
\begin{eqnarray*}
	& S(U_{a_{l}}^\T HV_{a_{l}})={\rm Diag}(h_{a_{l}})+O(\|H\|^{2}), \quad l=1,\ldots,r ,&
	\\[3pt]
	&S(U_{b_1}^\T HV_{b_1})={\rm Diag}(h_{b_1})+O(\|H\|^{2}),\quad \left[ U_{b_2}^\T HV_{b_2}\quad U_{b_2}^\T HV_{2} \right]=\left[ {\rm Diag}(h_{b_2})\quad 0 \right]+O(\|H\|^{2}).
	\nonumber
	&
\end{eqnarray*}
Therefore, since ${\cal C}=g'(\sigma)-{\rm Diag}(\eta)$, by (\ref{eq:R-k-semismooth-spectral-op-2}) and (\ref{eq:R-k-semismooth-spectral-op-3}), we obtain from (\ref{eq:Delta2-semismooth-spectral-op}) that
\begin{equation}\label{eq:Delta-semismooth-spectral-op}
 {\Delta(H)}
=U\left[{\rm Diag}\left(g'(\sigma)h\right)  \quad 0\right]V^\T +O(\|H\|^{2}) .
\end{equation}

On the other hand, for $X$ sufficiently close to $\overline{X}$, we have ${\cal U}_{l}(X)={\sum_{i\in a_{l}}}u_{i}v_{i}^\T $, $l=1,\ldots,r$. Therefore,
\begin{equation}\label{eq:GR-semismooth-spectral-op-1}
G_{R}(X)=G(X)-G_{S}(X)=\sum_{l=1}^{r}\sum_{i\in a_{l}}[g_i(\sigma)-g_i(\overline{\sigma})]u_{i}v_{i}^\T +\sum_{i\in b}g_i(\sigma)u_iv_i^\T  .
\end{equation}
 {Note that by definition, $G_R(\overline{X}) = 0$.}
 We know from \cite[Theorem 4]{DSSToh18} that $G$ is differentiable at $X$ if and only if $g$ is differentiable at $\sigma$.  Since $g$ is $1$-order G-semismooth at $\overline{\sigma}$ and $\sigma(\cdot)$ is strongly semismooth, we obtain that for any $X\in{\cal D}_{G}\cap{\cal B}$ (shrinking ${\cal B}$ if necessary),
\begin{equation*}
g(\sigma)-g(\overline{\sigma})=g'(\sigma)(\sigma-\overline{\sigma})+O(\|H\|^{2})=g'(\sigma)(h+O(\|H\|^{2}))+O(\|H\|^{2})=g'(\sigma)h+O(\|H\|^{2}) .
\end{equation*} Then, since $U\in{\mathbb O}^{m}$ and $U\in{\mathbb O}^{n}$ are uniformly bounded, we obtain from (\ref{eq:GR-semismooth-spectral-op-1}) that
\[
G_{R}(X)=U\left[{\rm Diag}\left(g'(\sigma)h\right)  \quad 0\right]V^\T +O(\|H\|^{2}) .
\] Thus, from (\ref{eq:Delta-semismooth-spectral-op}), we obtain that $\Delta(H)=G_{R}(X)+O(\|H\|^{2})$. That is, for any $X\in{\cal D}_{G}$ converging to $\overline{X}$,
\[
G(X)-G(\overline{X})-G'(X)H =  {G_R(X) + G_S(X) - G_S(\overline{X})  - G'(X)H}
= G_{R}(X) -\Delta(H)=O(\|H\|^{2}) .
\]

$``\Longrightarrow"$ Suppose that $G$ is $1$-order G-semismooth at $\overline{X}$.  Let $(\overline{U},\overline{V})\in{\mathbb O}^{m\times n}(\overline{X})$ be fixed.  Assume that $\sigma=\overline{\sigma}+h\in {\cal D}_{g}$ and $h\in\R^{m}$ is sufficiently small. Let $X=\overline{U}\left[{\rm Diag}(\sigma)\quad 0 \right] \overline{V}^\T $ and $H=\overline{U}\left[{\rm Diag}(h)\quad 0 \right] \overline{V}^\T $. Then, $X\in{\cal D}_{G}$ and converges to $\overline{X}$ if $h$ goes to zero. We know from \cite[Proposition 3]{DSSToh18} that for all $h$ sufficiently close to $0$, $G(X)=\overline{U}{\rm Diag}(g({\sigma}))\overline{V}_1^\T$. Therefore,   for any $h$ sufficiently close to $0$,
\[
{\rm Diag}(g(\overline{\sigma}+h)-g(\overline{\sigma}))= \overline{U}^\T  \left(G(X)-G(\overline{X})\right)\overline{V}_1=\overline{U}^\T G'(X)H\overline{V}_1+O(\|H\|^{2}) .
\]
Hence, since obviously ${\rm Diag}(g'(\sigma)h)=\overline{U}^\T G'(X)H\overline{V}_1$, we know that for $h$ sufficiently small,
$
g(\overline{\sigma}+h)-g(\overline{\sigma})=g'(\overline{\sigma})h+O(\|h\|^{2})
$.   {Thus,} $g$ is $1$-order G-semismooth at $\overline{\sigma}$.
\end{proof}

It is worth mentioning that for matrix optimization problems, we are able to obtain the semismoothness of the proximal point mapping $P_{f}$ defined by \eqref{eq:def-proximal} by employing the corresponding results on tame functions. We first recall the following concept on the {\it o(rder)-minimal structure} (cf. \cite[Definition 1.4]{Coste99}).

\begin{definition}\label{def:o-minimal-structure}
	An o-minimal structure of $\mathbb{R}^{n}$ is a sequence $ {\cal M}=\{ {\cal M}_{i}\}_{i=1}^{\infty}$
	 {such that} for each $i\ge 1$, ${\cal M}_{i}$ is a collection of subsets of $\mathbb{R}^{i}$ satisfying the following axioms.
	\begin{itemize}
		\item[(i)] For every $i$, ${\cal M}_{i}$ is closed under Boolean operators (finite unions, intersections and complement).
		\item[(ii)] If $A\in {\cal M}_{i}$ and $B\in{\cal M}_{i'}$, then $A\times B$ belongs to ${\cal M}_{i+i'}$.
		\item[(iii)] ${\cal M}_{i}$ contains all the subsets of the form $\{x\in\mathbb{R}^{i}\mid p(x)=0\}$, where $p:\mathbb{R}^{i}\to\mathbb{R}$ is a polynomial function.
		\item[(iv)] Let $\Pi: \mathbb{R}^{i+1}\to\mathbb{R}^{i}$ be the projection on the first $i$ coordinates. If $A\in{\cal M}_{i+1}$, then $\Pi(A)\in{\cal M}_{i}$.
		\item[(v)] The elements of ${\cal M}_{1}$ are exactly the finite union of points and intervals.
	\end{itemize}
	The elements of o-minimal structure are called definable sets. A map $F: A\subseteq\mathbb{R}^{n}\to\mathbb{R}^{m}$ is called definable if its graph is a definable subset of $\mathbb{R}^{n+m}$.
\end{definition}
A set of $\mathbb{R}^{n}$ is called  {\it tame} with respect to an o-minimal structure, if its intersection with the interval $[-r,r]^{n}$ for every $r > 0$ is definable in this structure, i.e., the element of this structure. A mapping is tame if its graph is tame. One most  {frequently} used o-minimal structure is the class of  semialgebraic subsets of $\mathbb{R}^{n}$. A set in $\mathbb{R}^{n}$ is {\it semialgebraic} if it is a finite union of sets of the form
\[
\left\{ x\in\mathbb{R}^{n}\,|\, p_{i}(x)>0,\ q_{j}(x)=0,\quad i=1,\ldots,a,\ j=1,\ldots,b \right\}\,,
\]
where $p_{i}:\mathbb{R}^{n}\to\mathbb{R}$, $i=1,\ldots,a$ and $q_{j}:\mathbb{R}^{n}\to\mathbb{R}$, $j=1,\ldots,b$ are polynomials. A mapping is semialgebraic if its graph is  semialgebraic.

For tame functions, we have the following proposition  of the semismoothness \cite{BDLewis09,Ioffe08}.
\begin{proposition}\label{prop:tame-semismooth}
	Let $\xi:\mathbb{R}^{n}\to\mathbb{R}^{m}$ be a locally Lipschitz continuous mapping.
	\begin{itemize}
		\item[(i)] If $\xi$ is tame, then $\xi$ is semismooth.
		
		\item[(ii)] If $\xi$ is semialgebraic, then $\xi$ is $\gamma$-order semismooth with some $\gamma> 0$.
	\end{itemize}
\end{proposition}

Let ${\cal Z}$ be a finite dimensional Euclidean space. If the closed proper convex function $f:{\cal Z}\to(-\infty,\infty]$ is semialgebraic, then the Moreau-Yosida regularization $\psi_{f}(x):=\displaystyle\min_{z\in {\cal Z}} \left\{ f(z)+\frac{1}{2}\|z-x\|^{2}\right\} $, $x\in{\cal Z}$ of $f$
is semialgebraic. Moreover, since the graph of the corresponding proximal point mapping $P_{f}$ is of the form
\[
{\rm gph}\,P_{f}= \left\{ (x,z)\in{\cal Z}\times{\cal Z}\,|\, f(z)+\frac{1}{2}\|z-x\|^{2}=\psi_{f}(x)\right\}\,,
\]
we know that $P_{f}$ is also semialgebraic (cf. \cite{Ioffe08}). Since $P_{f}$ is globally Lipschitz continuous, according to Proposition \ref{prop:tame-semismooth} (ii), it yields that $P_{f}$ is $\gamma$-order semismooth with some $\gamma> 0$. On the other hand, most unitarily invariant closed proper convex functions $f:{\cal X}\to(-\infty,\infty]$ in MOPs are semialgebraic. For example, it is easy to verify that the indicator function $\delta_{\mathbb{S}_{+}^{n}}(\cdot)$ of the positive semidefinite (PSD) matrix cone  and the matrix Ky Fan $k$-norm $\|\cdot\|_{(k)}$ (the sum of $k$-largest singular values of matrices) are all semialgebraic. Therefore, we know that the corresponding  proximal point mapping $P_{f}$ defined by \eqref{eq:def-proximal} for MOPs are $\gamma$-order semismooth with some $\gamma> 0$. However,  {since$\gamma$ is
not known explicitly,} by this approach, we may not  {be able to} show the strong semismoothness of  {the} spectral operator $G=P_f$ even  {if} the corresponding symmetric mapping $g$ is strongly semismooth.

\section{Characterization of Clarke's generalized Jacobian}\label{section:Clake g-J}

Let  $\overline{X}\in{\cal N}$ be given. In this section, we assume that $g$ is locally Lipschitz continuous near $\overline{\sigma}=\sigma(\overline{X})$ and directionally differentiable at $\overline{\sigma}$. Therefore, from Theorem \ref{thm:Lip-spectral-op} and \cite[Theorem 3 and Remark 1]{DSSToh18}, we know that the corresponding spectral operator $G$ is locally Lipschitz continuous near $\overline{X}$ and  directionally differentiable at $\overline{X}$. Furthermore, we define the function $d:\R^{m}\to\R^{m}$ by
\begin{equation}\label{eq:def_d}
d(h):=g(\overline{\sigma}+h)-g(\overline{\sigma})-g'(\overline{\sigma};h),\quad h\in\R^{m} .
\end{equation}
Consequently, we know that the function $d$ is also a mixed symmetric mapping, with respect to
$ \mathbb{P}^{|a_1|}\times\ldots\times\mathbb{P}^{|a_r|}\times\pm\mathbb{P}^{|b|}$,
over ${\cal V}=\R^{|a_1|}\times\ldots\times\R^{|a_r|}\times\R^{|b|}$.
Again, since $g$ is locally Lipschitz continuous near $\overline{\sigma}$ and  {directionally} differentiable at $\overline{\sigma}$, we know from \cite{Shapiro90} that $g$ is B-differentiable at $\overline{\sigma}$. Thus,  $d$ is differentiable at zero with the derivative $d'(0)=0$. Furthermore, if we assume that the function $d$ is also strictly differentiable at zero, then we have
\begin{equation}\label{eq:theta-strict-diff}
\lim_{w,w'\to 0\atop w\neq w'}\frac{d(w)-d(w')}{\|w-w'\|}=0 .
\end{equation} By using the mixed symmetric property of $d$, one can easily obtain the following results. We omit the details of the proof here.

\begin{lemma}\label{lem:limit-deri-theta}
	Let $d:\R^m\to\R^m$ be the function given by \eqref{eq:def_d}. Suppose that $d$ is strictly differentiable at zero. Let $\{w^k\}$ be a given sequence in $\R^m$ converging to zero. Then, if there exist $i,j\in a_{l}$ for some $l\in\{1,\ldots,r\}$ or $i, j\in b$ such that $w^k_i\neq w^k_j$ for all $k$ sufficiently large, then
	\begin{equation}\label{eq:limit-deri-theta-1}
	\lim_{k\to \infty}\frac{d_i(w^k)-d_j(w^k)}{w^k_i-w^k_j}=0 ;
	\end{equation} if there exist $i,j \in b$ such that $w^k_i+w^k_j\neq 0$ for all $k$ sufficiently large, then
	\begin{equation}\label{eq:limit-deri-theta-2}
	\lim_{k\to \infty}\frac{d_i(w^k)+d_j(w^k)}{w^k_i+w^k_j}=0 ;
	\end{equation}
	and if there exists $i\in b$ such that $w^k_i\neq 0$ for all $k$ sufficiently large, then
	\begin{equation}\label{eq:limit-deri-theta-3}
	\lim_{k\to \infty}\frac{d_i(w^k)}{w^k_i}=0 .
	\end{equation}
\end{lemma}

Again, since the spectral operator $G$ is locally Lipschitz continuous near $\overline{X}$, we know that $\Psi=G'(\overline{X};\cdot)$ is globally Lipschitz continuous (cf. \cite[Theorem A.2]{Robinson87} or \cite[Lemma 2.2]{QSun93}). Therefore, $\partial_B \Psi(0)$ and $\partial \Psi(0)$ are well-defined. Furthermore, we have the following characterization of   the B-subdifferential   and Clarke's subdifferential
of the spectral operator $G$ at $\overline{X}$.

\begin{theorem}\label{thm:B-subdiff-spectral-op}
	Suppose that the given $\overline{X}\in{\cal N}$ has the decomposition (\ref{eq:Y-eig-Z-SVD}). Suppose that there exists an open neighborhood ${\cal B}\subseteq\R^{m}$ of $\overline{\sigma}$ in $\hat{\sigma}_{\cal N}$ such that $g$ is differentiable at $\sigma\in{\cal B}$ if and only if $g'(\overline{\sigma};\cdot)$ is differentiable at $\sigma-\overline{\sigma}$. Assume further that the function $d:\R^{m}\to\R^{m}$ defined by \eqref{eq:def_d} is strictly differentiable at zero. Then, we have
	\[
	\partial_{B}G(\overline{X})=\partial_{B}\Psi(0) \quad {\rm and} \quad \partial G(\overline{X})=\partial\Psi(0) .
	\]
\end{theorem}
\begin{proof}
	We only need to prove the result for the B-subdifferentials. Let ${\cal V}$ be any element of $\partial_{B}G(\overline{X})$. Then, there exists a sequence $\{X^k\}$ in ${\cal D}_{G}$ converging to $\overline{X}$ such that ${\cal V}=\displaystyle{\lim_{k\to\infty}} G'(X^k)$.
	 {Now we present two preparatory steps before proving
	that ${\cal V} \in \partial_B \Psi(0)$.}
	
	\noindent{ (a)}
	For each $X^k$, let $U^k\in{\mathbb O}^{m}$ and $V^k\in{\mathbb O}^{n}$ be the  matrices such that
\[
X^k=U^k[\Sigma(X^k)\quad 0](V^k)^\T  .
\] For each $X^k$, denote $\sigma^k=\sigma(X^k)$. Then, we know from \cite[Theorem 4]{DSSToh18} that for each $k$, $\sigma^{k}\in{\cal D}_{g}$. For $k$ sufficiently large,  we know from \cite[Lemma 1]{DSSToh18} that for each $k$, $G_{S} $ is twice continuously differentiable at $\overline{X}$. Thus,  $\displaystyle{\lim_{k\to\infty}}G'_{S}(X^k)=G'_{S}(\overline{X})$. Hence, we have for any $H\in\V^{m\times n}$,
\begin{equation}\label{eq:limit-GS-spectral-op}
\lim_{k\to \infty}G'_{S}(X^k)H=G'_{S}(\overline{X})H=\overline{U}\left[\overline{\cal E}^0_1\circ S(\overline{U}^\T H\overline{V}_{1})+\overline{\cal E}^0_2\circ T(\overline{U}^\T H\overline{V}_{1})\quad \overline{\cal F}^0\circ \overline{U}^\T H\overline{V}_{2}  \right]\overline{V}^\T  .
\end{equation}
Moreover, we know that the mapping $G_{R}=G-G_{S}$ is also differentiable at each $X^k$ for $k$ sufficiently large. Therefore, we have
\begin{equation}\label{eq:U-Gs-Gr}
{\cal V}=\lim_{k\to \infty}G'(X^k)=G'_{S}(\overline{X})+\lim_{k\to \infty}G_{R}'(X^k) .
\end{equation}
From the continuity of the singular value function $\sigma(\cdot)$, by taking a subsequence if necessary, we assume that for each $X^k$ and $l,l'\in\{1,\ldots,r\}$, $\sigma_i(X^k)>0$, $\sigma_{i}(X^k)\neq \sigma_{j}(X^k)$ for any $i\in a_{l}$, $j\in a_{l'}$ and $l\neq l'$. Since $\{U^{k}\}$ and $\{V^{k}\}$ are uniformly bounded, by taking subsequences if necessary, we may also assume that  $\{U^{k}\}$ and $\{V^{k}\}$ converge and denote the limits by $U^{\infty}\in{\mathbb O}^{m}$ and $V^{\infty}\in{\mathbb O}^{n}$, respectively. It is clear that $(U^{\infty},V^{\infty})\in{\mathbb O}^{m,n}(\overline{X})$. Therefore, we know from \cite[Proposition 5]{DSToh10} that there exist $Q_{l}\in{\mathbb O}^{|a_{l}|}$, $l=1,\ldots,r$,  $Q'\in{\mathbb O}^{|b|}$ and $Q''\in{\mathbb O}^{n-|a|}$ such that $U^{\infty}=\overline{U}M$ and $V^{\infty}=\overline{V}N$, where $M={\rm Diag}(Q_{1},\ldots,Q_{r},Q')\in{\mathbb O}^{m}$ and $N={\rm Diag}(Q_{1},\ldots,Q_{r},Q'')\in{\mathbb O}^{n}$. Let $H\in\V^{m\times n}$ be arbitrarily given. For each $k$, denote $\widetilde{H}^k:=(U^k)^\T HV^k$. Since $\{(U^k,V^k)\}\in{\mathbb O}^{m,n}(X^k)$ converges to $(U^{\infty},V^{\infty})\in{\mathbb O}^{m,n}(\overline{X})$, we know that ${\lim_{k\to\infty}}\widetilde{H}^k=(U^{\infty})^\T HV^{\infty}$.
For  notational simplicity, we denote $\widetilde{H}:=\overline{U}^\T H\overline{V}$ and $\widehat{H}:=(U^{\infty})^\T HV^{\infty}$.

For $k$ sufficiently large, we know from \cite[Proposition 8]{DSToh10} and \cite[(38) in Theorem 4]{DSSToh18} that for any $H\in\V^{m\times n}$, $G_R'(X^k)H=U^k\Delta^k(V^k)^\T $ with
\[
\Delta^k:=\left[\begin{array}{cc}
{\rm Diag}\left(\Delta_1^k ,\dots,\Delta_{r}^k  \right) & 0\\
0  & \Delta_{r+1}^k
\end{array}\right]\in\V^{m\times n} ,
\]
where for each $k$, $\Delta_l^k=({\cal E}_{1}(\sigma^k))_{a_{l}a_{l}}\circ S(\widetilde{H}^k_{a_{l}a_{l}})+{\rm Diag}(({\cal C}(\sigma){\rm diag}(S(\widetilde{H}^k)))_{a_{l}})$, $l=1,\ldots,r$,
\[
\Delta_{r+1}^k=\left[({\cal E}_{1}(\sigma^k))_{bb}\circ S(\widetilde{H}^k_{bb})+{\rm Diag}(({\cal C}(\sigma^k){\rm diag}(S(\widetilde{H}^k)))_{b})+({\cal E}_{2}(\sigma^k))_{bb}\circ T(\widetilde{H}^k_{bb})\quad ({\cal F}(\sigma^k))_{bc}\circ \widetilde{H}^k_{bc}\right]
\]
and ${\cal E}_{1}(\sigma^k)$, ${\cal E}_{2}(\sigma^k)$, ${\cal F}(\sigma^k)$ and ${\cal C}(\sigma^k)$ are defined for $\sigma^k$ by \cite[(34)--(36)]{DSSToh18}, respectively. Again, since $\{U^{k}\}$ and $\{V^{k}\}$ are uniformly bounded, we know that
\begin{equation}\label{eq:Gr'H-limit}
\lim_{k\to \infty}G_{R}'(X^k)H=U^{\infty}(\lim_{k\to\infty}\Delta^k)(V^{\infty})^\T =\overline{U}M(\lim_{k\to\infty}\Delta^k)N^\T \overline{V}^\T  .
\end{equation}

\noindent {(b)}
For each $k$, denote $w^k:=\sigma^k-\overline{\sigma}\in\R^{m}$. Moreover, for each $k$, we can define $W^k_l:=Q_l{\rm Diag}(w^k_{a_l})Q_l^\T \in\S^{|a_l|}$, $l=1,\ldots,r$ and $W^k_{r+1}:=Q'[{\rm Diag}(w^k_b)\quad 0]Q''^\T \in\V^{|b|\times(n-|a|)}$. Therefore, it is clear that for each $k$, $W^k:=(W^k_1,\ldots,W^k_l,W^k_{r+1})\in{\cal W}$ and $\kappa(W^k)=w^k$, where ${\cal W}=\S^{|a_{1}|}\times\ldots\times\S^{|a_{r}|}\times\V^{|b|\times(n-|a|)}$. Moreover, since ${\lim_{k\to \infty}}\sigma^k=\overline{\sigma}$, we know that ${\lim_{k\to \infty}}W^k=0$ in ${\cal W}$.
From the assumption, we know that $\phi=g'(\overline{\sigma};\cdot)$ and $d(\cdot)$ are differentiable at each $w^k$ and $\phi'(w^{k})=g'(\sigma^{k})-d'(w^{k})$ for all $w^{k}$. Since $d$ is strictly differentiable at zero, it can be checked easily that $\lim_{k\to\infty}d'(w^{k})=d'(0)=0$. By taking a subsequence if necessary, we may assume that $\lim_{k\to\infty}g'(\sigma^k)$ exists. Therefore, we have
\begin{equation}\label{eq:lim-g'-phi'}
\lim_{k\to\infty}\phi'(w^k)=\lim_{k\to\infty}g'(\sigma^k) .
\end{equation}
Since $\Phi$ is the spectral operator with respect to the mixed symmetric mapping $\phi$, from \cite[Theorem 7]{DSSToh18} we know that $\Phi$ is differentiable at $W\in{\cal W}$  if and only if $\phi$ is differentiable at $\kappa(W)$.
Recall that  ${\widehat \Phi}: {\cal W}\to\V^{m\times n}$ is defined by (\ref{eq:def-Diag-Phi}). Then, for $k$ sufficiently large, $ {\widehat \Phi}$ is differentiable at $W^k$. Moreover, for each $k$, we  define the matrix $C^k\in\V^{m\times n}$ by
\[
C^k=\overline{U} \left[\begin{array}{cc} {\rm Diag}\left( W^k_{1},\dots,W^k_{r}\right) & 0 \\[2mm]
0 & W^k_{r+1}
\end{array}\right]
\overline{V}^\T  .
\]  Then, we know that for $k$ sufficiently large, $\Psi $ is differentiable at $C^k$ and ${\lim_{k\to \infty}}C^k=0$ in $\V^{m\times n}$.
Thus, we know from \eqref{eq:def-Psi} that for each $k$,
\[
\Psi'(C^k)H=G'_S(\overline{X})H+\overline{U}\left[{\widehat \Phi}'(W^k)D(H)\right]\overline{V}^\T
\quad\forall \; H\in\V^{m\times n}  ,
\]
where $D(H)=\left(S(\widetilde{H}_{a_{1}a_{1}}),\ldots,S(\widetilde{H}_{a_{r}a_{r}}), \widetilde{H}_{b\bar{a}} \right)$ with $\widetilde{H}=\overline{U}^\T H\overline{V}$ and ${\widehat \Phi^\prime}(W^k)D(H)$ can be derived from \cite[Theorem 7]{DSSToh18}.
By comparing with \eqref{eq:U-Gs-Gr} and \eqref{eq:Gr'H-limit},
we know that
 { ${\cal V}\in\partial_B \Psi(0)$ if we can} show that
\begin{equation}\label{eq:K=limRt-1}
\lim_{k\to\infty}\Delta^k=\lim_{k\to\infty}M^\T {\widehat \Phi}'(W^k)D(H)N .
\end{equation}

 {To show that (\ref{eq:K=limRt-1}) holds, we consider eight different cases.}
For any $(i,j)\in\{1,\ldots,m\}\times\{1,\ldots,n\}$, consider the following cases.

{\bf  Case 1:} $i=j$. It is easy to check that for each $k$,
\[
(\Delta^k)_{ii}=(g'(\sigma^k)h^k)_i \quad {\rm and}\quad  \left(M^\T {\widehat \Phi}'(W^k)D(H)N\right)_{ii}=(\phi'(w^k)\widehat{h})_i ,
\]
where $h^k=\left({\rm diag}(S(\widetilde{H}^{k}_{aa})),{\rm diag}(\widetilde{H}^{k}_{bb})\right)$ and $\widehat{h}=\left({\rm diag}(S(\widehat{H}_{aa})),{\rm diag}(\widehat{H}_{bb})\right)$. Therefore, we know from \eqref{eq:lim-g'-phi'} that
\[
\lim_{k\to\infty}(\Delta^k)_{ii}=\lim_{k\to\infty}(g'(\sigma^k)h^k)_i=\lim_{k\to\infty}(\phi'(w^k)\widehat{h})_i=\lim_{k\to\infty}\left(M^\T {\widehat \Phi}'(W^k)D(H)N\right)_{ii} .
\]

{\bf  Case 2:} $i,j\in a_{l}$ for some $l\in\{1,\ldots,r\}$, $i\neq j$ and $\sigma^k_{i}\neq\sigma^k_{j}$ for $k$ sufficiently large. We obtain that for $k$ sufficiently large,
\begin{eqnarray*}
	&(\Delta^k)_{ij}=\frac{g_i(\sigma^k)-g_j(\sigma^k)}{\sigma^k_{i}-\sigma^k_{j}}(S(\widetilde{H}^{k}_{a_la_l}))_{ij},&
	\\
	&
	\left(M^\T {\widehat \Phi}'(W^k)D(H)N\right)_{ij}=\frac{\phi_i(w^k)-\phi_j(w^k)}{w^k_{i}-w^k_{j}}(S(\widehat{H}_{a_la_l}))_{ij}  .
	&
\end{eqnarray*}
Since $\overline{\sigma}_{i}=\overline{\sigma}_{j}$ and $g_{i}(\overline{\sigma})=g_{j}(\overline{\sigma})$, we know that for $k$ sufficiently large,
\begin{eqnarray}
\frac{g_i(\sigma^k)-g_j(\sigma^k)}{\sigma^k_{i}-\sigma^k_{j}}&=&\frac{g_i(\overline{\sigma}+w^k)-g_j(\overline{\sigma}+w^k)}{w^k_i-w^k_j}=\frac{g_i(\overline{\sigma}+w^k)-g_i(\overline{\sigma})+g_j(\overline{\sigma})-g_j(\overline{\sigma}+w^k)}{w^k_i-w^k_j}\nonumber\\
&=&\frac{d_i(w^k)-d_j(w^k)}{w^k_i-w^k_j}+\frac{\phi_i(w^k)-\phi_j(w^k)}{w^k_i-w^k_j} .\label{eq:difference-g-theta-phi-1}
\end{eqnarray} Therefore, we know from \eqref{eq:limit-deri-theta-1} that
\[
\lim_{k\to\infty}\frac{g_i(\sigma^k)-g_j(\sigma^k)}{\sigma^k_{i}-\sigma^k_{j}}(S(\widetilde{H}^{k}_{a_la_l}))_{ij}=\lim_{k\to\infty}\frac{\phi_i(w^k)-\phi_j(w^k)}{w^k_{i}-w^k_{j}}(S(\widehat{H}_{a_la_l}))_{ij} ,
\] which implies $\displaystyle{\lim_{k\to\infty}}(\Delta^k)_{ij}=\displaystyle{\lim_{k\to\infty}}\left(M^\T {\widehat \Phi}'(W^k)D(H)N\right)_{ij}$.

{\bf  Case 3:} $i,j\in a_{l}$ for some $l\in\{1,\ldots,r\}$, $i\neq j$ and $\sigma^k_{i}= \sigma^k_{j}$ for $k$ sufficiently large. We have for $k$ sufficiently large,
\begin{eqnarray*}
	&(\Delta^k)_{ij}=\left((g'(\sigma^k))_{ii}-(g'(\sigma^k))_{ij}\right)(S(\widetilde{H}^{k}_{a_la_l}))_{ij},&
	\\[3pt]
	&\left(M^\T {\widehat \Phi}'(W^k)D(H)N\right)_{ij}=\left((\phi'(w^k))_{ii}-(\phi'(w^k))_{ij}\right)(S(\widehat{H}_{a_la_l}))_{ij}  .&
\end{eqnarray*}
Therefore, we obtain from \eqref{eq:lim-g'-phi'} that
\[
\lim_{k\to\infty}\left((g'(\sigma^k))_{ii}-(g'(\sigma^k))_{ij}\right)(S(\widetilde{H}^{k}_{a_la_l}))_{ij}=\lim_{k\to\infty}\left((\phi'(w^k))_{ii}-(\phi'(w^k))_{ij}\right)(S(\widehat{H}_{a_la_l}))_{ij} .
\]
Thus, we have $\displaystyle{\lim_{k\to\infty}}(\Delta^k)_{ij}=\displaystyle{\lim_{k\to\infty}}\left(M^\T {\widehat \Phi}'(W^k)D(H)N\right)_{ij}$.

{\bf  Case 4:} $i,j\in b$, $i\neq j$ and $\sigma^k_{i}=\sigma^k_{j}>0$ for $k$ sufficiently large.
We have for $k$ large,
\begin{eqnarray*}
	&(\Delta^k)_{ij}=\left((g'(\sigma^k))_{ii}-(g'(\sigma^k))_{ij}\right)(S(\widetilde{H}^{k}_{bb}))_{ij}
	+\frac{g_i(\sigma^k)+g_j(\sigma^k)}{\sigma^k_{i}+\sigma^k_{j}}(T(\widetilde{H}^{k}_{bb}))_{ij},&
	\\[3pt]
	&\left(M^\T {\widehat \Phi}'(W^k)D(H)N\right)_{ij}=\left((\phi'(w^k))_{ii}-(\phi'(w^k))_{ij}\right)(S(\widehat{H}_{bb}))_{ij}+\frac{\phi_i(w^k)+\phi_j(w^k)}{w^k_{i}+w^k_{j}}(T(\widehat{H}_{bb}))_{ij} .&
\end{eqnarray*}
Since $\overline{\sigma}_{i}=\overline{\sigma}_{j}=0$ and $g_i(\overline{\sigma})=g_j(\overline{\sigma})=0$, we get
\begin{eqnarray}
\frac{g_i(\sigma^k)+g_j(\sigma^k)}{\sigma^k_{i}+\sigma^k_{j}}
&=&\frac{d_i(w^k)+d_j(w^k)}{w^k_i+w^k_j}+\frac{\phi_i(w^k)+\phi_j(w^k)}{w^k_i+w^k_j} .\label{eq:difference-g-theta-phi-2}
\end{eqnarray}
Therefore, we know from (\ref{eq:limit-deri-theta-2}) and (\ref{eq:lim-g'-phi'}) that $\displaystyle{\lim_{k\to\infty}}(\Delta^k)_{ij}=\displaystyle{\lim_{k\to\infty}}\left(M^\T {\widehat \Phi}'(W^k)D(H)N\right)_{ij}$.

{\bf  Case 5:} $i,j\in b$, $i\neq j$ and $\sigma^k_{i}\neq\sigma^k_{j}$ for $k$ sufficiently large. For large $k$, we have
\begin{eqnarray*}
	&(\Delta^k)_{ij}=\frac{g_i(\sigma^k)-g_j(\sigma^k)}{\sigma^k_{i}-\sigma^k_{j}}(S(\widetilde{H}^{k}_{bb}))_{ij}+
	\frac{g_i(\sigma^k)+g_j(\sigma^k)}{\sigma^k_{i}+\sigma^k_{j}}(T(\widetilde{H}^{k}_{bb}))_{ij},&
	\\[3pt]
	&\left(M^\T {\widehat \Phi}'(W^k)D(H)N\right)_{ij}=\frac{\phi_i(w^k)-\phi_j(w^k)}{w^k_{i}-w^k_{j}}(S(\widehat{H}_{bb}))_{ij}+\frac{\phi_i(w^k)+\phi_j(w^k)}{w^k_{i}+w^k_{j}}(T(\widehat{H}_{bb}))_{ij} .&
\end{eqnarray*}
Thus, by (\ref{eq:difference-g-theta-phi-1}) and (\ref{eq:difference-g-theta-phi-2}), we know from (\ref{eq:limit-deri-theta-1}) and (\ref{eq:limit-deri-theta-2}) that $\displaystyle{\lim_{k\to\infty}}(\Delta^k)_{ij}=\displaystyle{\lim_{k\to\infty}}\left(M^\T {\widehat \Phi}'(W^k)D(H)N\right)_{ij}$.

{\bf  Case 6:} $i,j\in b$, $i\neq j$ and $\sigma^k_{i}=\sigma^k_{j}=0$ for $k$ sufficiently large. We know  {that for $k$ sufficiently large},
\begin{eqnarray*}
	&(\Delta^k)_{ij}=\left((g'(\sigma^k))_{ii}-(g'(\sigma^k))_{ij}\right)(S(\widetilde{H}^{k}_{bb}))_{ij}+
	(g'(\sigma^k))_{ii}(T(\widetilde{H}^{k}_{bb}))_{ij},
	&
	\\[3pt]
	&\left(M^\T {\widehat \Phi}'(W^k)D(H)N\right)_{ij}=\left((\phi'(w^k))_{ii}-(\phi'(w^k))_{ij}\right)(S(\widehat{H}_{bb}))_{ij}+(\phi'(w^k))_{ii}(T(\widehat{H}_{bb}))_{ij} .&
\end{eqnarray*}
Again, we obtain from \eqref{eq:lim-g'-phi'} that $\displaystyle{\lim_{k\to\infty}}(\Delta^k)_{ij}=\displaystyle{\lim_{k\to\infty}}\left(M^\T {\widehat \Phi}'(W^k)D(H)N\right)_{ij}$.

{\bf  Case 7:} $i\in b$, $j\in c$ and $\sigma^k_{i}>0$ for $k$ sufficiently large. We have for $k$ sufficiently large,
\begin{equation*}
(\Delta^k)_{ij}=\frac{g_i(\sigma^k)}{\sigma^k_{i}}(\widetilde{H}^{k}_{bc})_{ij},
\quad \left(M^\T {\widehat \Phi}'(W^k)D(H)N\right)_{ij}=\frac{\phi_i(w^k)}{w^k_{i}}(\widehat{H}_{bc})_{ij} .
\end{equation*}
Since $\overline{\sigma}_{i}=0$ and $g_{i}(\overline{\sigma})=0$, we get
\[
\frac{g_i(\sigma^k)}{\sigma^k_{i}}=\frac{g_i(\overline{\sigma}+w^k)-g_i(\overline{\sigma})}{w^k_{i}}=\frac{d_i(w^k)}{w^k_{i}}+\frac{\phi_i(w^k)}{w^k_{i}} .
\] Therefore,  by (\ref{eq:limit-deri-theta-3}), we obtain that $\displaystyle{\lim_{k\to\infty}}(\Delta^k)_{ij}=\displaystyle{\lim_{k\to\infty}}\left(M^\T {\widehat \Phi}'(W^k)D(H)N\right)_{ij}$.

{\bf  Case 8:} $i\in b$, $j\in c$ and $\sigma^k_{i}=0$ for $k$ sufficiently large. We have
for $k$ sufficiently large,
\begin{eqnarray*}
	&(\Delta^k)_{ij}=(g'(\sigma^k))_{ii}(\widetilde{H}^{k}_{bc})_{ij},\quad
	\left(M^\T {\widehat \Phi}'(W^k)D(H)N\right)_{ij}=(\phi'(w^k))_{ii}(\widehat{H}_{bc})_{ij} .&
\end{eqnarray*}
Therefore, by (\ref{eq:lim-g'-phi'}), we obtain that $\displaystyle{\lim_{k\to\infty}}(\Delta^k)_{ij}=\displaystyle{\lim_{k\to\infty}}\left(M^\T {\widehat \Phi}'(W^k)D(H)N\right)_{ij}$.

Thus, we know that \eqref{eq:K=limRt-1} holds. Therefore, by \eqref{eq:U-Gs-Gr} and \eqref{eq:Gr'H-limit}, we obtain that ${\cal V}\in\partial_{B}\Psi(0)$.

\medskip
Conversely, suppose that ${\cal V}\in\partial_{B}\Psi(0)$ is arbitrarily chosen. Then, from the definition of $\partial_{B}\Psi(0)$, we know that there exists a sequence $\{C^k\}\subseteq\V^{m\times n}$ converging to zero such that $\Psi$ is differentiable at each $C^k$ and ${\cal V}={\lim_{k\to\infty}}\Psi'(C^k)$. For each $k$, we know from \eqref{eq:def-Psi} that $\Psi$ is differentiable at $C^k$ if and only if the spectral operator $\Phi:{\cal W}\to {\cal W}$ is differentiable at $W^k:=D(C^k)=\left(S(\widetilde{C}^k_{a_{1}a_{1}}),\ldots,S(\widetilde{C}^k_{a_{r}a_{r}}),\widetilde{C}^k_{b\bar{a}}\right)\in{\cal W}$, where for each $k$, $\widetilde{C}^k=\overline{U}^\T C^k\overline{V}$.
Moreover, for each $k$, we have the following decompositions
\begin{eqnarray*}
	S(\widetilde{C}^k_{a_{l}a_{l}})=Q^k_{l}\Lambda(S(\widetilde{C}^k_{a_{l}a_{l}}))(Q^k_{l})^\T ,\; l=1,\ldots,r,
	\quad
	\widetilde{C}^k_{b\bar{a}}={Q'}^k\left[ \Sigma(\widetilde{C}^k_{b\bar{a}})\quad 0 \right]({Q''}^k)^\T  ,
\end{eqnarray*}
where $Q^k_{l}\in{\mathbb O}^{|a_{l}|}$,   ${Q'}^k\in{\mathbb O}^{|b|}$ and ${Q''}^k\in{\mathbb O}^{n-|a|}$. For each $k$,
let
\begin{eqnarray*}
	&w^k:=\left(\lambda(S(\widetilde{C}^k_{a_{1}a_{1}})),\ldots,\lambda(S(\widetilde{C}^k_{a_{r}a_{r}})),
	\sigma(\widetilde{C}^k_{b\bar{a}})\right)\in\R^m,&
	\\
	&M^k:={\rm Diag}\Big(Q_1^k,\dots,Q_{r}^k,  {Q^\prime}^k\Big) \in{\mathbb O}^m,
	\quad
	N^k:= {\rm Diag}\Big(Q_1^k,\dots,Q_{r}^k,  {Q^{\prime\prime}}^k\Big)
	\in{\mathbb O}^n.&
\end{eqnarray*}
Since $\{M^k\}$ and $\{N^k\}$ are uniformly bounded, by taking subsequences if necessary, we know that there exist $Q_l\in{\mathbb O}^{|a_l|}$, $Q'\in{\mathbb O}^{|b|}$ and $Q''\in{\mathbb O}^{n-|b|}$ such that
\[
\lim_{k\to\infty}M^k=M:=
{\rm Diag}\Big(Q_1,\dots,Q_{r},  {Q^\prime}\Big)
\quad
\lim_{k\to\infty}N^k=N:=
{\rm Diag}\Big(Q_1,\dots,Q_{r},  {Q^{\prime\prime}}\Big).
\]
For each $k$, by  \cite[Theorem 7]{DSSToh18}, we know that for any $H\in\V^{m\times n}$,
\begin{equation}\label{eq:Psi'CkH}
\Psi'(C^k)H = \overline{U}\left[\overline{\cal E}^0_1\circ S(\overline{U}^\T H\overline{V}_{1})+\overline{\cal E}^0_2\circ T(\overline{U}^\T H\overline{V}_{1})\quad \overline{\cal F}^0\circ \overline{U}^\T H\overline{V}_{2}  \right]\overline{V}^\T + \overline{U}\left[{\widehat \Phi}'(W^k)D(H)\right]\overline{V}^\T,
\end{equation}
where $D(H)=\left(S(\widetilde{H}_{a_{1}a_{1}}),\ldots,S(\widetilde{H}_{a_{r}a_{r}}), \widetilde{H}_{b\bar{a}} \right)$ with $\widetilde{H}=\overline{U}^\T H\overline{V}$. Let $R^k: =\Phi'_k(W^k)D(H)$,  $k=1, \ldots,  r+1$.

For each $k$, define $\sigma^k:=\overline{\sigma}+w^k\in\R^m$.
Since ${\lim_{k\to\infty}}w^k=0$ and for each $k$, $w^k_i\geq 0$ for all $i\in b$,
we have $\sigma^k\geq 0$ for $k$ sufficiently large.
Therefore, for $k$ sufficiently large, we are able to define
\[
X^k:=\overline{U}M[{\rm Diag}(\sigma^k)\quad 0]N^\T \overline{V}^\T \in\V^{m\times n} .
\]
For simplicity, denote $U=\overline{U}M\in{\mathbb O}^{m}$ and $ V=\overline{V}N\in{\mathbb O}^{n}$. It is clear that the sequence $\{X^k\}$ converges to $\overline{X}$. From the assumption, we know that $g$ is differentiable at each $\sigma^k$ and $d$ is differentiable at each $w^{k}$ with $g'(\sigma^{k})=\phi'(w^{k})+d'(w^{k})$ for all $\sigma^{k}$. Therefore, by \cite[Theorem 4]{DSSToh18}, we know that $G$ is differentiable at each $X^k$. By taking subsequences if necessary, we may assume that $\lim_{k\to\infty}\phi'(w^{k})$ exists. Thus, since $d$ is strictly differentiable at zero,  we know that \eqref{eq:lim-g'-phi'} holds.
Since the derivative formula \eqref{eq:Gs'-formula} is independent of $(U,V)\in{\mathbb O}^{m,n}(\overline{X})$, we know from \cite[(38) in Theorem 4]{DSSToh18} that for any $H\in\V^{m\times n}$,
\begin{eqnarray}
G'(X^k)H&=&\overline{U}\left[\overline{\cal E}^0_1\circ S(\overline{U}^\T H\overline{V}_{1})+\overline{\cal E}^0_2\circ T(\overline{U}^\T H\overline{V}_{1})\quad \overline{\cal F}^0\circ \overline{U}^\T H\overline{V}_{2}  \right]\overline{V}^\T \nonumber\\
&&+\;\overline{U}
\left[\begin{array}{cc} {\rm Diag}\left(Q_1\Omega_1^kQ_1^\T,\dots,
Q_r\Omega_r^kQ_r^\T\right)  & 0 \\[2mm]
0 & Q'\Omega_{r+1}^kQ''^\T  \end{array}\right]
\overline{V}^\T   ,
\label{eq:G'm-form-conver}
\end{eqnarray}
where for each $k$, $\Omega_l^k=({\cal E}_{l}(\sigma^k))_{a_{l}a_{l}}\circ S(\widehat{H}_{a_{l}a_{l}})+{\rm Diag}(({\cal C}(\sigma^k){\rm diag}(S(\widehat{H})))_{a_{l}})$, $l=1,\ldots,r$ and
\[
\Omega_{r+1}^k=\left[({\cal E}_{1}(\sigma^k))_{bb}\circ S(\widehat{H}_{bb})+{\rm Diag}(({\cal C}(\sigma^k){\rm diag}(S(\widehat{H})))_{b})+({\cal E}_{2}(\sigma^k))_{bb}\circ T(\widehat{H}_{bb})\quad ({\cal F}_{2}(\sigma^k))_{bc}\circ \widehat{H}_{bc}\right] ,
\]
${\cal E}_{1}(\sigma^k)$, ${\cal E}_{2}(\sigma^k)$ and ${\cal F}(\sigma^k)$ are defined   by \cite[(34)--(36)]{DSSToh18}, respectively and $\widehat{H}:=M^\T \overline{U}^\T H\overline{V}N=M^\T \widetilde{H}N$. Therefore, by comparing \eqref{eq:Psi'CkH} and \eqref{eq:G'm-form-conver}, we know that the inclusion ${\cal V}\in \partial_B G(\overline{X})$ follows if we  {can} show that
\begin{equation}\label{eq:lim-R-Omega}
\lim_{k\to \infty} \left(R_1^k,\dots, R_r^k, R_{r+1}^k \right)
=\lim_{k\to \infty}
\left(Q_1\Omega_1^kQ_1^\T,\dots,Q_r\Omega_r^kQ_r^\T,
Q'\Omega_{r+1}^kQ''^\T \right) .
\end{equation}
 {Similar} to the proofs  {for} {\bf  Cases 1-8} in the first part, by using   (\ref{eq:lim-g'-phi'}) and \eqref{eq:limit-deri-theta-1}--\eqref{eq:limit-deri-theta-3} in Lemma \ref{lem:limit-deri-theta}, we can show that \eqref{eq:lim-R-Omega} holds. For simplicity, we omit the details here. Therefore, we obtain that $\partial_{B}G(\overline{X})=\partial_{B}\Psi(0)$. This completes the proof.
\end{proof}

\section{Extensions}\label{section:extension}

 {In this section, we consider the extensions of the related results obtained in previous sections for the case that ${\cal X}\equiv \mathbb{V}^{m\times n}$ to the  general spectral operators defined on the vector space ${\cal X}$ given by \eqref{eq:space-X}, i.e., the Cartesian product of several  real or complex matrices. One special class of this nature are the smoothing spectral operators.}

\subsection{The spectral operators defined on the general matrix spaces}

In fact, the corresponding properties of the general spectral operators defined on the vector space ${\cal X}$ given by \eqref{eq:space-X}, including locally Lipschitzian continuity, $\rho$-order B-differentiability, $\rho$-order G-semismoothness and
the characterization of the Clarke  generalized Jacobian, can be studied in the same fashion as those in Sections
\ref{section:B-diff}--\ref{section:Clake g-J}. For simplicity, we omit the
proofs here. For readers who are interested in seeking the  {details}, we refer them to \cite{Ding12}.

Let ${\cal X}$ and ${\cal Y}$ be the vector spaces defined by \eqref{eq:space-X} and \eqref{eq:space-Y}, respectively. Suppose that ${\cal N}$  {is}  a given nonempty open set in ${\cal X}$. Let $G:{\cal X}\to{\cal X}$ be the spectral operator defined in Definition \ref{def:def-spectral-op} with respect to $g:{\cal Y}\to{\cal Y}$, which is mixed symmetric on an open set $\hat{\kappa}_{\cal N}$ in ${\cal Y}$ containing $\kappa_{\cal N}:=\{\kappa(X)\mid X\in {\cal N}\}$. For the given $\overline{X}=(\overline{X}_1,\ldots,\overline{X}_{s_0},\overline{X}_{s_0+1},\ldots,\overline{X}_{s})\in{\cal X}$, recall  {that} $\kappa(\overline{X})=\left(\lambda(\overline{X}_1),\ldots,\lambda(\overline{X}_{s_0}),\sigma(\overline{X}_{s_0+1}),\ldots,\sigma(\overline{X}_s)\right)\in{\cal Y}$. We first consider the locally Lipschitzian continuity of spectral operators of matrices.
\begin{theorem}\label{thm:Lip-spectral-op-block}
	Let $\overline{X}\in{\cal N}$ be given. The spectral operator $G$ is locally Lipschitz continuous near $\overline{X}$ if and only if the corresponding mixed symmetric function $g$ is locally Lipschitz continuous near $\kappa(\overline{X})$.
\end{theorem}

For the $\rho$-order B(ouligand)-differentiability ($0<\rho\leq 1$) of the general spectral operators, we have the following theorem.
\begin{theorem}\label{thm:rho-order-B-diff-spectral-op-block}
	Let $\overline{X}\in{\cal N}$ and $0<\rho\leq 1$ be given. Then, we have the following results.
	\begin{itemize}
		\item[(i)] If $g$ is locally Lipschitz continuous near $\kappa(\overline{X})$ and $\rho$-order B-differentiable at $\kappa(\overline{X})$, then $G$ is $\rho$-order B-differentiable at $\overline{X}$.
		\item[(ii)] If $G$ is $\rho$-order B-differentiable at $\overline{X}$, then $g$ is $\rho$-order B-differentiable at $\kappa(\overline{X})$.
	\end{itemize}
\end{theorem}

Suppose that $g$ is locally Lipschitz continuous near $\kappa(\overline{X})$.  {Then} we know from Theorem \ref{thm:Lip-spectral-op-block} that the corresponding spectral operator $G$ is also locally Lipschitz continuous near $\overline{X}$. We have the following theorem on the G-semismoothness of spectral operators.

\begin{theorem}\label{thm:rho-order-G-semismooth-spectral-op-block}
	Let $\overline{X}\in{\cal N}$ be given. Suppose that $0<\rho\leq 1$. Then, the spectral operator $G$ is $\rho$-order G-semismooth at $\overline{X}$ if and only if  $g$ is $\rho$-order G-semismooth at $\kappa(\overline{X})$.
\end{theorem}

Finally, we assume that $g$ is locally Lipschitz continuous near $\overline{\kappa}=\kappa(\overline{X})$ and directionally differentiable at $\overline{\kappa}$. From Theorem \ref{thm:Lip-spectral-op-block}, \cite[Theorems 6 and Remark 1]{DSSToh18}, the spectral operator $G$ is also locally Lipschitz continuous near $\overline{X}$ and directionally differentiable at $\overline{X}$. Then, we have the following results on the characterization of  the Clarke generalized Jacobian of $G$.

\begin{theorem}\label{thm:B-subdiff-spectral-op-block}
	Let $\overline{X}\in{\cal N}$ be given. Suppose that there exists an open neighborhood ${\cal B}\subseteq{\cal Y}$ of $\overline{\kappa}$ in $\hat{\kappa}_{\cal N}$ such that $g$ is differentiable at $\kappa\in{\cal B}$ if and only if $\phi=g'(\overline{\kappa};\cdot)$ is differentiable at $\kappa-\overline{\kappa}$. Assume that the function $d:{\cal Y}\to{\cal Y}$ defined by
	\[
	d(h)=g(\overline{\kappa}+h)-g(\overline{\kappa})-g'(\overline{\kappa};h),\quad h\in{\cal Y}
	\] is strictly differentiable at zero. Then, we have
	\[
	\partial_{B}G(\overline{X})=\partial_{B}\Psi(0)\quad {\rm and} \quad \partial G(\overline{X})=\partial\Psi(0) ,
	\]where $\Psi:=G'(\overline{X};\cdot):{\cal X}\to{\cal X}$ is the directional derivative of $G$ at $\overline{X}$.
\end{theorem}

\subsection{The smoothing spectral operators}

In this subsection, we consider the smoothing spectral operators of matrices. For simplicity, we mainly focus on the case ${\cal X}\equiv \R\times \mathbb{V}^{m\times n}$. The corresponding results can be obtained as special cases for the spectral operators defined on the general matrix space $\cal X$ given by \eqref{eq:space-X}.

Let ${\cal N}$ be a given nonempty open set in $\V^{m\times n}$. Suppose that $g:\R^m\to\R^m$ is  mixed symmetric with respect to ${\cal P}\equiv\pm{\mathbb P}^m$ on an open set $\hat{\sigma}_{ {\cal N}}$ in $\R^{m}$ containing $\sigma _{\cal N}=\left\{\sigma(X)\mid X\in{\cal N}\right\}$. Let $\overline{X}\in {\cal N}$ be given. Assume that $g$ is Lipschitz continuous near $\overline{\sigma}=\sigma(\overline{X})$. Suppose there exists a mapping $\theta:\mathbb{R}_{++}\times\hat{\sigma}_{ {\cal N}}\to\mathbb{R}^m$
 {such that} for any $x\in \hat{\sigma}_{ {\cal N}}$ and $(\omega,z)\in\mathbb{R}_{++}\times\hat{\sigma}_{ {\cal N}}$ close to $(0,x)$, $\theta$ is continuously differentiable  {around}  $(\omega,z)$ unless $\omega=0$ and $\theta(\omega,z)\to g(x)$ as $(\omega,z)\to(0,x)$. For convenience, for any $x\in \hat{\sigma}_{ {\cal N}}$, we always define $\theta(0,x)=g(x)$  and $\theta(\omega,x)=\theta(-\omega,x)$ for any $\omega<0$. Furthermore, we assume that for any fixed $\omega$ close to $0$, $\theta(\omega,\cdot)$ is also  mixed symmetric on $\hat{\sigma}_{ {\cal N}}$. Then, the mapping $\theta$ is said to be a smoothing approximation of $g$ on $\hat{\sigma}_{ {\cal N}}$.  For a given mixed symmetric mapping $g$, there are many ways to construct such a smoothing approximation. For example, as mentioned in Section \ref{section:Lip}, the Steklov averaged function defined  by \eqref{eq:def-Steklov averaged function} is a smoothing approximation of the  mixed symmetric mapping $g$.

 {Define $\pi: \mathbb{R}\times\hat{\sigma}_{ {\cal N}}\to \mathbb{R}\times\mathbb{R}^m$ by $\pi(\omega,x)=(\omega, \theta(\omega,x))$, $(\omega,x)\in \mathbb{R}\times\hat{\sigma}_{ {\cal N}}$. Then, it is easy to verify that $\pi$ is mixed symmetric (Definition \ref{def:mixed_symmetric}) over $\mathbb{R}\times\mathbb{R}^m$ with respect to $\pm\mathbb{P}^{1}\times\pm \mathbb{P}^{m}$. Note that $\mathbb{R}\equiv\mathbb{V}^{1\times 1}$. The  spectral operator $\Pi:\mathbb{V}^{1\times 1}\times\mathbb{V}^{m\times n}\to \mathbb{V}^{1\times 1}\times\mathbb{V}^{m\times n}$ defined  with respect to $\pi$ takes the form:
	$$\Pi(\omega,X)=(\omega,\Theta(\omega,X)),\quad (\omega,X)\in \mathbb{V}^{1\times 1}\times{\cal N},$$
	where $\Theta(\omega,X):=U\left[{\rm Diag}\big(\theta(\omega,\sigma(X))\big)\quad 0\right]V^\T$ and $(U,V)\in \mathbb{O}^{m,n}(X)$.} We call $\Theta:\mathbb{V}^{1\times 1}\times{\cal N}\to\mathbb{V}^{m\times n}$ the smoothing spectral operator of $G$ with respect to $\theta$. It follows from \cite[Theorem 1]{DSSToh18} that $\Theta$ is well-defined. Moreover, since $\theta$ is continuously differentiable at any $(\omega,z)\in\mathbb{R}\times\hat{\sigma}_{ {\cal N}}$ with $\omega$ close to $0$, we know from \cite[Theorem 7]{DSSToh18} that $\Theta$ is also continuously differentiable at any $(\omega,X)\in\mathbb{R}\times{\cal N}$, and the corresponding derivative formula can be found in \cite[Theorem 7]{DSSToh18}. For the case $\omega=0$, the continuity and Hadamard  {directional} differentiability of $\Theta$ follows directly from \cite[Theorem 6]{DSSToh18}.  {Next, we study the locally Lipschitz continuity, $\rho$-order B-differentiable ($0<\rho\leq 1$), $\rho$-order G-semismooth ($0<\rho\le 1$), and  the characterization of  the Clarke generalized Jacobian of $\Theta$ at $(0,\overline{X})$.  
The first property we consider is  the local  Lipschitzian continuity of $\Theta$ near $(0,\overline{X})$.}

 {\begin{theorem}\label{thm:Lip-smoothing-spectral}
	Let $\overline{X}\in {\cal N}$ be given. Suppose that the smoothing approximation $\theta$ of $g$ is locally Lipschitz continuous near $(0,\overline{\sigma})$. Then, the smoothing spectral operator $\Theta$ with respect to $\theta$ is locally Lipschitz continuous near $(0,\overline{X})$.
\end{theorem}
}

The following theorem is on the $\rho$-order B-differentiability ($0<\rho\leq 1$) of the smoothing spectral operator $\Theta$ at $(0,\overline{X})$.
\begin{theorem}\label{thm:rho-order-B-diff-smoothing-spectral}
	Let $\overline{X}\in{\cal N}$ and $0<\rho\leq 1$ be given. If the smoothing approximation $\theta$ of $g$ is locally Lipschitz continuous near $(0,\overline{\sigma})$ and $\rho$-order B-differentiable at $(0,\overline{\sigma})$, then the smoothing spectral operator $\Theta$ is $\rho$-order B-differentiable at $(0,\overline{X})$.
\end{theorem}

Suppose that the smoothing approximation $\theta$ of $g$ is locally Lipschitz continuous near $(0,\sigma(\overline{X}))$. Then,   {by Theorem \ref{thm:Lip-smoothing-spectral}},  the smoothing spectral operator $\Theta$ is also locally Lipschitz continuous near $\overline{X}$. Moreover, we have  the following results on the G-semismoothness of the smoothing spectral operator $\Theta$ at $(0,\overline{X})$.

\begin{theorem}
	Let $\overline{X}\in {\cal N}$ be given. Suppose that the smoothing approximation $\theta$ of $g$ is $\rho$-order G-semismooth ($0<\rho\le 1$) at $(0,\sigma(\overline{X}))$. Then, the corresponding smoothing spectral operator $\Theta$ is $\rho$-order G-semismooth at $(0,\overline{X})$.
\end{theorem}

Finally, suppose that the smoothing approximation $\theta$ of $g$ is locally Lipschitz continuous near $(0,\overline{\sigma})$ and directionally differentiable at $(0,\overline{\sigma})$. It then follows from Theorem \ref{thm:Lip-smoothing-spectral} and \cite[Theorems 3]{DSSToh18} that the smoothing  spectral operator $\Theta$ is also locally Lipschitz continuous near $(0,\overline{X})$ and directionally differentiable at $(0,\overline{X})$. Furthermore, we have the following results on the characterization of  the Clarke generalized Jacobian of $\Theta$ at $(0,\overline{X})$.

\begin{theorem}\label{thm:B-subdiff-smoothing-spectral}
	Let $\overline{X}\in{\cal N}$ be given. Suppose that there exists an open neighborhood ${\cal B}\subseteq\mathbb{R}\times\hat{\sigma}_{\cal N}$ of $(0,\overline{\sigma})$ such that $\theta$ is differentiable at $(\tau,\sigma)\in{\cal B}$ if and only if $\theta'((0,\overline{\sigma});(\cdot,\cdot))$ is differentiable at $(\tau,\sigma-\overline{\sigma})$. Assume that the function $d:\mathbb{R}\times\mathbb{R}^m\to\mathbb{R}^m$ defined by
	\[
	d(\tau,h):=\theta(\tau,\overline{\sigma}+h)-\theta(0,\overline{\sigma})-\theta'((0,\overline{\sigma});\tau,h),\quad (\tau,h)\in\mathbb{R}\times\mathbb{R}^m
	\] is strictly differentiable at zero. Then, we have
	\[
	\partial_{B}\Theta(0,\overline{X})=\partial_{B}\Psi(0,0)\quad {\rm and} \quad \partial \Theta(0,\overline{X})=\partial\Psi(0,0) ,
	\]
	where $\Psi:=\Theta'((0,\overline{X});(\cdot,\cdot))$ is the directional derivative of $\Theta$ at $(0,\overline{X})$.
\end{theorem}

\section{Conclusions}\label{sect:remarks}
In this paper, we conduct extensive studies  on spectral operators   initiated in  \cite{DSSToh18}. Several fundamental first and second-order properties of spectral operators, including the
locally Lipschitz continuity,   $\rho$-order B(ouligand)-differentiability ($0<\rho\leq 1$),   $\rho$-order G-semismooth ($0<\rho\leq 1$) and the characterization of Clarke's generalized Jacobian are systematically studied.
These results, together with the results obtained in \cite{DSSToh18} provide the necessary theoretical foundations for both the computational and
theoretical aspects of many applications. In particular,  based on the recent exciting progress made in solving  {large scale} SDP problems, we believe  {that} the properties of the spectral operators  {studied here}, such as the semismoothness and the characterization of Clarke's generalized Jacobian, constitute the backbone for future developments on both designing some efficient numerical methods for solving large-scale MOPs and  conducting second-order variational analysis of the general MOPs. The work done on spectral operators of matrices is by no means complete. Due to the rapid advances in the applications of matrix optimization in different fields, spectral operators of matrices will become even more important and many other properties of spectral operators are waiting to be explored.



\begin{thebibliography}{999}
	{\scriptsize
	\bibitem{ABSvaiter13}
	{\sc H. Attouch, J. Bolte, and B. F. Svaiter},
	{\em Convergence of descent methods for semi-algebraic
		and tame problems: proximal algorithms,
		forward–backward splitting, and regularized
		Gauss–Seidel methods},
	{Mathematical Programming}, {137} (2013), pp. 91-129.

	
	\bibitem{Bhatia97}
	{\sc R. Bhatia},
	{\em Matrix Analysis}, Springer, New York, 1997.
	



	\bibitem{BCRoy87}
	{\sc J. Bochnak, M. Coste, and M.-F. Roy}, {\em Real Algebraic Geometry}, volume 36 of Ergebnisse derMathematik
	und ihrer Grenzgebiete (3) [Results in Mathematics and Related Areas (3)]. Springer, Berlin, 1998. Translated from the 1987 French original, Revised by the authors.
	
	\bibitem{BDLewis09}
	{\sc J. Bolte, A. Daniilidis, and A. S. Lewis},
	{\em Tame functions are semismooth}.
	{Mathematical Programming}, {117} (2009), pp. 5-19.
	
	\bibitem{BHPauwels18}
	{\sc J. Bolte, A. Hochart, and E. Pauwels},
	{\em Qualification conditions in semialgebraic programming},
	{SIAM Journal of Optimization}, {28} (2018), pp. 1867-1891.
	
	\bibitem{BPauwels16}
	{\sc J. Bolte and E. Pauwels},
	 {\em Majorization-minimization procedures and convergence of
		SQP methods for semi-algebraic and tame programs},
	{Mathematics of Operations Research}, {41} (2016), pp. 442-465.
	

	\bibitem{CanRec08}
	{\sc E. J. Cand\`{e}s and B. Recht},
	{\em Exact matrix completion via convex optimization},
	{Foundations of Computational Mathematics}, {9} (2008), pp. 717-772.
	
	
	
	\bibitem{CanTao09}
	{\sc E. J. Cand\`{e}s and T. Tao},
	{\em The power of convex relaxation: near-optimal matrix completion},
	{IEEE Transactions on Information Theory}, {56} (2009), pp. 2053-2080.
	
	
	\bibitem{CLMW09}
	{\sc E. J. Cand\`{e}s, X. Li, Y. Ma, and J. Wright},
	{\em Robust principal component analysis?}
	{Journal of the ACM}, {58} (2011), article No. 11.
	

	
	\bibitem{CSun08}
	{\sc Z. X. Chan and D. F. Sun},
	{\em Constraint nondegeneracy, strong regularity, and nonsingularity in semidefinite programming},
	{SIAM Journal on Optimization}, {19} (2008), pp. 370-396.
	
	\bibitem{CSPW09}
	{\sc V. Chandrasekaran, S. Sanghavi, P. A. Parrilo, and A. Willsky},
	{\em Rank-sparsity incoherence for matrix decomposition},
	{SIAM Journal of Optimization}, {21} (2011), pp. 572-596.
	
	\bibitem{CLSToh12}
	{\sc C. H. Chen, Y. J. Liu, D. F.  Sun, and K. C. Toh},
	{\em A semismooth Newton-CG dual proximal point algorithm for matrix spectral norm approximation problems},
	{Mathematical Programming},  {155} (2016), pp. 435-470.
	

	
	
	\bibitem{CFP03}
	{\sc M. Chu, R. Funderlic, and R. Plemmons},
	{\em Structured low rank approximation},
	{Linear Algebra and its Applications}, {366} (2003), pp. 157-172.
	
	\bibitem{CHien15}
	{\sc C. B. Chua and L. T. K. Hien},
	{\em A superlinearly convergent smoothing Newton continuation algorithm for variational inequalities over definable sets},
	{SIAM Journal on Optimization}, {25} (2015), pp. 1034-1063.
	
	
	\bibitem{Clarke83}
	{\sc F. H. Clarke},
	{\em Optimization and Nonsmooth Analysis}, John Wiley \& Sons, New York, 1983.
	
	\bibitem{Coste99}
	{\sc M. Coste},
	{\em An Introduction to o-minimal Geometry}, RAAG Notes, Institut de
	Recherche Math\'ematiques de Rennes, 1999.
	
	\bibitem{CDZhao17}
	{\sc Y. Cui, C. Ding, and X.-Y. Zhao},
	 {\em  Quadratic growth conditions for convex matrix optimization problems associated with spectral functions},
	 {SIAM Journal on Optimization}, {27} (2017), pp. 2332-2355.
	

	
	\bibitem{Ding12}
	{\sc C. Ding},
	{\em An Introduction to a Class of Matrix Optimization Problems}, PhD thesis,
	National University of Singapore, available at \url{http://www.mypolyuweb.hk/~dfsun/DingChao_Thesis_final.pdf}, 2012.
	
	\bibitem{DSSToh18}
	{\sc C. Ding, D. F. Sun, J. Sun, and K. C. Toh},
	{\em Spectral operators of matrices}, {Mathematical Programming}, {168} (2018), pp. 509-531.
	
	\bibitem{DSToh10}
	{\sc C. Ding, D. F. Sun, and K. C. Toh},
	{\em An introduction to a class of matrix cone programming},
	{Mathematical Programming}, {144} (2014), pp. 141-179.
	
	\bibitem{DSYe14}
	{\sc C. Ding, D. F. Sun, and J. J. Ye},
	{\em First order optimality conditions for mathematical programs with semidefinite cone complementarity constraints},
	{Mathematical Programming}, {147} (2014), pp. 539-579.
	

	


	
	\bibitem{Dobrynin04}
	{\sc V. Dobrynin},
	{\em On the rank of a matrix associated with a graph},
	{Discrete Mathematics}, {276} (2004), pp. 169-175.
	
	\bibitem{DIoffe15}
	{\sc D. Drusvyatskiy and A. D.  Ioffe},
	{\em Quadratic growth and critical point stability
		of semi-algebraic functions},
	{Mathematical Programming}, {153} (2015), pp. 635-653.
	
	\bibitem{DILewis15}
	{\sc D. Drusvyatskiy, A. D. Ioffe, and A. S. Lewis},
	{\em Curves of descent},
	{SIAM Journal on Control and Optimization}, {53} (2015), pp. 114-138.
	

	
	\bibitem{FPang03}
	{\sc F. Facchinei  and J. S. Pang},
	{\em Finite-dimensional variational inequalities and complementarity problems}, Springer-Heidelberg, New York, 2003.
	
	\bibitem{FPSTseng13}
	{\sc M. Fazel, T. K. Pong, D. Sun, and P. Tseng},
	 {\em Hankel matrix rank minimization with applications in system identification and realization}, SIAM Journal on Matrix Analysis and Applications, 34 (2013), pp. 946-977.
	
	


	\bibitem{GVanLoan12}
	{\sc G. H. Golub and C. F. Van Loan},
	{\em Matrix Computations}, 4th edition. Johns Hopkins University Press, Baltimore, MD, 2012.
	
	\bibitem{GT94}
	{\sc A. Greenbaum and L. N. Trefethen},
	{\em GMRES/CR and Arnoldi/Lanczos as matrix approximation problems},
	{SIAM Journal on Scientific Computing}, {15} (1994), pp. 359-368.
	

	\bibitem{Gupal77}
	{\sc A. M. Gupal}, {\em A method for the minimization of almost-differentiable functions}, Kibernetika, 1 (1977), pp. 114-116.
	

	
	\bibitem{HBenisrael73}
	{\sc J. B. Hawkins and A. Ben-Israel},
	{\em On generalized matrix functions}, Linear and Multilinear Algebra, 1 (1973), pp. 163-171.
	
	\bibitem{Higham08}
	{\sc N. J. Higham},
	{\em Functions of Matrices}, SIAM, Philadelphia, 2008.
	
	\bibitem{HJohnson94}
	{\sc R. A. Horn and C. R. Johnson},
	{\em Topics in Matrix Analysis}, Cambridge University Press, Cambridge, 1994.
	
	\bibitem{Ioffe08}
	{\sc A. Ioffe},
	{\em An invitation to tame optimization},
	SIAM Journal on Optimization, {19} (2009), pp. 1894-1917.
	
	\bibitem{IKSolodov13}
	{\sc A. F. Izmailov, A. S. Kurennoy, and M. V. Solodov }
	{\em The Josephy--Newton method for semismooth generalized equations and semismooth SQP for optimization},
Set-Valued and Variational Analysis 21 (2013), pp. 17-45.
	

	
	\bibitem{KLVempala97}
	{\sc A. Kotlov, L.  Lov\'{a}sz,  and S. Vempala},
	{\em The Colin de Verdi\`{e}re number and sphere representations of a graph},
	{Combinatorica}, {17} (1997), pp. 483--521.
	

	
	
	\bibitem{Lancaster64}
	{\sc P. Lancaster},
	{\em On eigenvalues of matrices dependent on a parameter},
	{Numerische Mathematik}, {6} (1964), pp. 377-387.
	


	
	\bibitem{LOverton96}
	{\sc A. S. Lewis and M. L. Overton},
	{\em Eigenvalue optimization},
	{Acta Numerica}, {5} (1996), pp. 149-190.

	
	\bibitem{LSendov05b}
	{\sc A. S. Lewis and H. S. Sendov},
	{\em Nonsmooth analysis of singular values. Part II: applications},
	{Set-Valued Analysis}, {13} (2005), pp. 243-264.
	
	\bibitem{LMPham15}
	{\sc G. Li, B. S. Mordukhovich, and T. S. Ph\d{a}m},
	{\em New fractional error bounds for polynomial systems
	with applications to H\"{o}lderian stability in optimization
	and spectral theory of tensors},
	{Mathematical Programming}, {153} (2015), pp. 333-362.
	
	\bibitem{LSToh18}
	{\sc X. Li, D. F. Sun,  and K. C. Toh},
	{\em A highly efficient semismooth  {Newton} augmented Lagrangian method for solving Lasso problems},
	SIAM Journal on Optimization, {28} (2018), pp. 433-458.
	
	
	
	\bibitem{LSToh18b}
	{\sc X. Li, D. F. Sun,  and K. C. Toh},
	{\em On efficiently solving the subproblems of a level-set method for fused lasso problems},
	SIAM Journal on Optimization, {28} (2018), pp. 1842-1866.
	
	\bibitem{LLCDai18}
	{\sc T. X. Liu, Z. S. Lu, X. J. Chen, and Y. H. Dai},
	{\em An exact penalty method for semidefinite-box constrained low-rank matrix optimization problems}, to appear IMA Journal of Numerical Analysis, (2018), pp, 1-22.
	
	\bibitem{LSToh12}
	{\sc Y. J. Liu, D. F. Sun,  and K. C. Toh},
	{\em An implementable proximal point algorithmic framework for nuclear norm minimization},
	{Mathematical Programming}, {133} (2012), pp. 399-436.
	
	\bibitem{Lowner34}
	{\sc K. L\"owner},
	{\em \"Uber monotone matrixfunktionen},
	{Mathematische Zeitschrift}, {38} (1934), pp. 177-216.
	
	\bibitem{Lovasz79}
	{\sc L. Lov\'{a}sz},
	{\em On the Shannon capacity of a graph},
	{IEEE Transactions on Information Theory}, {25} (1979), pp. 1--7.
	
	\bibitem{LZLi15}
	{\sc Z. S. Lu, Y. Zhang, and X. R. Li},
	{\em Penalty decomposition methods for rank minimization},
	Optimization Methods and Software, 30 (2015), pp. 531--558.
	
	\bibitem{LZLu17}
	{\sc Z. S. Lu, Y. Zhang, and J. Lu},
	{\em $\ell_p$ regularized low-rank approximation via iterative reweighed singular value minimization},
	Computational Optimization and Applications, 68 (2017), pp. 619-642.
	
	
	
	\bibitem{MSPan12}
	{\sc W.M. Miao, D. F. Sun, and S. H. Pan},
	{\em A rank-corrected procedure for matrix completion with fixed basis coefficients},
	{Mathematical Programming}, {159} (2016), pp. 289--338.
	
	
	\bibitem{Mifflin77}
	{\sc R. Mifflin},
	{\em Semismooth and semiconvex functions in constrained optimization},
	{SIAM Journal on Control and Optimization}, {15} (1977), pp. 959-972.
	
	
	
	\bibitem{MNRockafellar15}
	{\sc B. S. Mordukhovich, T. T. A. Nghia, and R. T. Rockafellar},
	{\em Full stability in finite-dimensional optimization},
	{Mathematics of Operations Research}, {40} (2015), pp. 226-252.
	
	\bibitem{MRockafellar12}
	{\sc B. S. Mordukhovich and R. T. Rockafellar},
	 {\em Second-order subdifferential calculus with applications to tilt stability in optimization}, SIAM Journal on Optimization, {22} (2012), pp. 953--986.


	\bibitem{Noferini17}
	{\sc V. Noferini},
	 {\em A formula for the Fr\'{e}chet derivative of a generalized matrix function},
	SIAM Journal on Matrix Analysis and Applications, {38} (2017), pp. 434-457.

	
	\bibitem{Pang90}
	{\sc J. S. Pang},
	{\em Newton's method for B-differentiable equations},
	{Mathematics of Operations Research}, 15 (1990), pp. 311-341.
	
	\bibitem{Pang91}
	{\sc J. S. Pang},
	{\em A B-differentiable equation-based, globally and locally quadratically convergent algorithm for nonlinear programs, complementarity and variational inequality problems},
	{Mathematical Programming}, 51 (1991), pp. 101-131.
	
	\bibitem{PQi93}
	{\sc J. S. Pang and L. Q. Qi},
	 {\em Nonsmooth equations: motivation and algorithms},
	 SIAM Journal on Optimization, {3} (1993), pp. 443-465.
	
	\bibitem{PSSun03}
	{\sc J. S. Pang, D. F. Sun, and J. Sun},
	{\em Semismooth homeomorphisms and strong stability of semidefinite and Lorentz complementarity problems},
	{Mathematics of Operations Research}, 28 (2002), pp. 39-63.
	
	
	\bibitem{Qi93}
	{\sc L. Qi},
	{\em Convergence analysis of some algorithms for solving nonsmooth equations},
	{Mathematics of Operations Research},  {18} (1993), pp. 227-244.
	
	\bibitem{QSun93}
	{\sc L. Qi and J. Sun},
	{\em A nonsmooth version of Newton's method},
	{Mathematical Programming}, {58} (1993), pp. 353-367.
	
	
	\bibitem{RFP07}
	{\sc B. Recht,  M. Fazel, and P. A. Parrilo},
	{\em Guaranteed minimum rank solutions to linear matrix equations via nuclear norm minimization},
	{SIAM Review}, {52} (2010), pp. 471-501.
	
	\bibitem{Robinson87}
	{\sc S. M. Robinson},
	{\em Local structure of feasible sets in nonlinear programming, Part III: stability and sensitivity},
	{Mathematical Programming Study}, {30} (1987), pp. 45-66.
	
	\bibitem{Rockafellar70}
	{\sc R. T. Rockafellar},
	{\em Convex Analysis}, Princeton University Press, Princeton, 1970.
	
	
	
	\bibitem{Shapiro90}
	{\sc A. Shapiro},
	{\em On concepts of directional differentiability},
	{Journal of Optimization Theory and Applications}, {66} (1990), pp. 477-487.
	
	
	\bibitem{Steklov1907}
	{\sc V. A. Steklov},
	{\em On the asymptotic representation of certain functions defined by a linear differential equation of the second order, and their application to the problem of expanding an arbitrary function into a series of these functions}, {Kharkov} (1957) (In Russian).

	
	\bibitem{Sun06}
	{\sc D. F. Sun},
	{\em The strong second order sufficient condition and constraint nondegeneracy in nonlinear semidefinite programming and their implications},
	{Mathematics of Operations Research}, {31} (2006), pp. 761-776.
	
	\bibitem{SSun02}
	{\sc D. F. Sun and J. Sun},
	{\em Semismooth matrix-valued functions},
	{Mathematics of Operations Research}, {27} (2002), pp. 150-169.
	
	\bibitem{SSun03}
	{\sc D. F. Sun and J. Sun},
	{\em Strong semismoothness of eigenvalues of symmetric matrices and its applications in inverse eigenvalue problems},
	{SIAM Journal on Numerical Analysis}, {40} (2003), pp. 2352-2367.
	
	\bibitem{SSun08}
	{\sc D. F. Sun and J. Sun},
	{\em L\"{o}wner's operator and spectral functions in Euclidean Jordan algebras},
	{Mathematics of Operations Research}, {33} (2008), pp. 421-445.

	
	\bibitem{Todd01}
	{\sc M.J. Todd},
	{\em Semidefinite optimization},
	{Acta Numerica}, {10} (2001), pp. 515-560.
	
	\bibitem{Toh97}
	{\sc K. C. Toh},
	{\em GMRES vs. ideal GMRES},
	{SIAM Journal on Matrix Analysis and Applications}, {18} (1997), pp. 30-36.
	
	\bibitem{TT98}
	{\sc K. C. Toh and L. N.  Trefethen},
	{\em The Chebyshev polynomials of a matrix},
	{SIAM Journal on Matrix Analysis and Applications}, {20} (1998), pp. 400-419.
	
	\bibitem{Torki01}
	{\sc M. Torki},
	{\em Second-order directional derivatives of all eigenvalues of a symmetric matrix},
	{Nonlinear Analysis}, {46} (2001), pp. 1133-1150.
	
	\bibitem{UUBrztzke17}
	{\sc M. Ulbrich, S. Ulbrich, and D. Bratzke},
	{\em A multigrid semismooth Newton method for semilinear contact problems},
	Journal of Computational Mathematics 35 (2017), pp. 484-526.
	

	

	
	\bibitem{WMGR09}
	{\sc J. Wright, Y.  Ma, A. Ganesh,  and S. Rao},
	{\em Robust principal component analysis: exact recovery of corrupted low-rank matrices via convex optimization},
	In Y. Bengio, D. Schuurmans,
	J. Lafferty, and C. Williams editors, {Advances in Neural Information Processing Systems}, {22} (2009).
	
	
	\bibitem{YSToh15}
	{\sc L. Q. Yang, D. F. Sun, and K. C. Toh},
	{\em SDPNAL$+$: a majorized semismooth Newton-CG augmented Lagrangian method for semidefinite programming with nonnegative constraints},
	{Mathematical Programming Computation}, {7} (2015), pp. 331-366.
	
	
	\bibitem{Yang09}
	{\sc Z. Yang},
	{\em A study on nonsymmetric matrix-valued functions}, Master's Thesis, National University of Singapore, \url{http://www.mypolyuweb.hk/~dfsun/Main_YZ.pdf}, 2009.
	
	\bibitem{YSToh18}
	{\sc Y. Yuan, D. F. Sun, and K. C. Toh},
	{\em An efficient semismooth Newton based algorithm for convex clustering}, Proceedings of the 35-th International Conference on Machine Learning (ICML), Stockholm, Sweden, PMLR 80, 2018.
	
	
	
	\bibitem{ZSToh10}
	{\sc X. Y. Zhao, D. F. Sun, and K. C. Toh},
	{\em A Newton-CG augmented Lagrangian method for semidefinite programming},
	{SIAM Journal on Optimization}, {20} (2010), pp. 1737-1765.
	
	\bibitem{ZSo17}
	{\sc Z. R. Zhou and A. M. C. So}
	{\em  A unified approach to error bounds for structured convex optimization problems}, Mathematical Programming, 165 (2017), pp. 689-728.
}
\end{thebibliography}


\end{document}